\documentclass{article}
\usepackage{graphicx} 
\usepackage[utf8]{inputenc}
\usepackage[T1]{fontenc}
\usepackage{amsmath}
\usepackage{amsthm}
\usepackage{amsfonts}
 \usepackage{mathtools}
 \usepackage{parskip}
\usepackage{tikz}
\usepackage{bbold}
\usepackage{comment}
\usepackage{arydshln}
\usepackage{float}
\usepackage{amssymb}
\usepackage{subcaption}
\usepackage[a4paper, total={5.7in, 8in}]{geometry}
\usepackage[version=4]{mhchem}
\usepackage{stmaryrd}
\usepackage{bbm}
\usepackage{color}
\usepackage{bm}
\usepackage[export]{adjustbox}
\usepackage{hyperref}
\usepackage{lineno}

\newtheorem{theorem}{Theorem}[section]
\newtheorem{lemma}[theorem]{Lemma}
\newtheorem{remark}[theorem]{Remark}
\theoremstyle{definition}
\newtheorem{definition}{Definition}[section]

\newcommand{\tn}{\tilde{n}}
\newcommand{\cha}{\mathbb{1}_{(-d_1,d_1)}}
\newcommand{\out}{\mathbb{1}_{\mathbb{R}\setminus (-d_1,d_1)}}
\newcommand{\tu}{\tilde{u}}
\newcommand{\ou}{\bu}
\newcommand{\bu}{\bar{u}}
\newcommand{\bbu}{\bar{U}}
\newcommand{\om}{{\Omega_{0}}}

\newcommand{\R}{{\mathbb{R}}}

\newcommand{\bydef}{\stackrel{\mbox{\tiny\textnormal{\raisebox{0ex}[0ex][0ex]{def}}}}{=}}

\title{Existence and orbital stability proofs of traveling wave solutions on an infinite strip for the suspension bridge equation}

\author{
Lindsey van der Aalst \footnote{VU Amsterdam, Department of Mathematics, De Boelelaan 1111, 1081 HV Amsterdam, The Netherlands.  {\tt l.j.w.van.der.aalst@vu.nl} (Corresponding author)}
\and
Matthieu Cadiot
\footnote{CMAP, CNRS, Ecole polytechnique, Institut Polytechnique de Paris, 91120 Palaiseau, France. {\tt matthieu.cadiot@polytechnique.edu}}
}

\begin{document}

\maketitle

\begin{abstract}
In this paper, we present a computer-assisted approach for constructively proving the existence of traveling wave solutions of the suspension bridge equation on the infinite strip $\Omega = \mathbb{R} \times (-d_2,d_2)$. Using a meticulous Fourier analysis, we derive a quantifiable approximate inverse $\mathbb{A}$ for the Jacobian $D\mathbb{F}(\bu)$ of the PDE at an approximate traveling wave solution $\bu$. Such approximate objects are obtained thanks to Fourier coefficient sequences and operators, arising from Fourier series expansions on a rectangle $\om = (-d_1,d_1) \times (-d_2,d_2)$ for large $d_1$. In particular, the challenging exponential nonlinearity of the equation is tackled using a rigorous control of the aliasing error when computing related Fourier coefficients. This allows to establish a Newton-Kantorovich approach, from which the existence of a true traveling wave solution of the PDE can be proven in a vicinity of $\bu$. We successfully apply such a methodology in the case of the suspension bridge equation and prove the existence of multiple traveling wave solutions on $\Omega$. Finally, given a proven solution $\tilde{u}$,  a Fourier series approximation on $\om$ allows us to accurately enclose the spectrum of $D\mathbb{F}(\tilde{u})$. Such a tight control provides the number of negative eigenvalues, which in turn, allows us to conclude about the orbital (in)stability of the traveling wave.
\end{abstract}

\begin{center}
\end{center}


\section{Introduction}
The objective of this paper is to prove the existence of traveling wave solutions of the two-dimensional suspension bridge equation on an infinite strip $\Omega \bydef \R \times (-d_2,d_2)$ for $d_2>0$ and to prove orbital (in)stability. We let $v=v(t,X_1,X_2) : \R^+ \times \Omega \to \R$ represent the deflection in the downward direction of the surface of the bridge and we let $\Delta^2=(\partial_{X_1}^2+\partial_{X_2}^2)^2$. One of the systems used for modeling suspension bridges is (e.g. \cite{SB_chen}, \cite{SB_nagatou})
\begin{equation}\label{eq:SBnowave}
    \partial_t^2 v = -\Delta^2 v  - e^{v} +1.
\end{equation}
Summaries of the historical development of different mathematical models for suspension bridges, both simplistic and advanced, are available in \cite{Drabek_overview} and \cite{Gazzola_overview}.

In the dynamics of the suspension bridge equation, traveling waves can be observed \cite{wave_breuer}. In order to look for such solutions, we introduce the parameter $c\in \mathbb{R}$ to denote the wave speed. We consider waves in the $X_1$-direction. Hence, we choose the traveling wave ansatz $v(t,X_1,X_2) = u(X_1-ct,X_2) = u(x_1,x_2)$ such that the function $u : \Omega \to \R$ satisfies
\begin{equation}\label{eq : suspension bridge equation}
    \Delta^2 u + c^2 \partial_{x_1}^2u + e^{u} -1 =0,
\end{equation}
with $u(x) \to 0$ as $|x_1| \to \infty$. In the $x_2$ direction, we impose the following boundary conditions \cite{Aalst2024}
\begin{align}\label{eq : neumann boundary conditions}
\partial_{x_2}u(\cdot,\pm d_2)=\partial_{x_2}^3u(\cdot,\pm d_2)=0.
\end{align}
We only consider values of $c$ in $(0,\sqrt{2})$, because if $c$ tends to zero, the amplitude of the waves grows to infinity, and for $c=0$, nontrivial solutions do not longer exist \cite{PeletierTroy}. Contrarily, if $c$ approaches $\sqrt{2}$, the solutions become oscillatory as the amplitudes of the waves tend to go to zero \cite{SB_horak}.

We want to establish the existence of a solution $\tilde{u}$ to \eqref{eq : suspension bridge equation} on $\Omega$. Numerical simulations of solutions to this equation, obtained via the Mountain Pass Algorithm, are discussed in \cite{SB_horak}. A rigorous existence proof for traveling wave solutions to \eqref{eq : suspension bridge equation} on finite domains, periodic in one direction and satisfying Neumann boundary conditions in the other, was later given in \cite{Aalst2024}. We combine the techniques and results outlined there with procedures developed in \cite{Cadiot2023} and \cite{Cadiot2024} to obtain an existence proof of a solution that decays to zero as $|x_1| \rightarrow  \infty$.

Proving the existence of solutions to nonlinear partial differential equations (PDEs) on unbounded domains poses substantial analytical challenges. In particular, the loss of compactness in the resolvent of differential operators on unbounded domains complicates the application of standard existence theorems. In \cite{BookOverviewCAP}, this obstacle is overcome by rigorously controlling the spectrum of the linearization of an approximate solution resulting in methods for proving weak solutions to PDEs of second and fourth order. These methods are demonstrated for the Schr\"{o}dinger equation on $\mathbb{R}^2$. That book also discusses the work of \cite{Pacella2016}, where solutions to Emden's equation on an unbounded L-shaped domain are proven. Additionally, \cite{Wunderlich_navierstokes} verifies the existence of a weak solution to the Navier-Stokes equations on a perturbed infinite strip. While methods providing weak solutions are advantageous for certain applications, they do not naturally extend to the proof of strong solutions in the present context. We note that while strong solutions on bounded domains are addressed in \cite{BookOverviewCAP}, our focus is on the construction of strong solutions on an infinite strip.

 In contrast, the previously mentioned papers \cite{Cadiot2023} and \cite{Cadiot2024} provide an approach for proving strong solutions on unbounded domains by introducing a new technique for constructing approximate inverses of PDE operators and proving compactness. In these works, solutions on unbounded domains are proven for equations with polynomial nonlinearities, specifically the Kawahara and Swift-Hohenberg equations. The techniques discussed in these papers need to be adjusted to be able to apply them to our suspension bridge equation that includes an exponential nonlinearity, which will be done in this paper. Our approach is versatile in the sense that it can be extended to other analytic nonlinearities as well.
 
 To deal with nonpolynomial terms such as the exponential term in \eqref{eq : suspension bridge equation}, the method described in \cite{Figueras} can be used to control the aliasing error. In \cite{JP_FFT}, it is demonstrated how to apply these techniques for working with nonpolynomial nonlinearities in one spatial dimension. The generalization to higher dimensions can be found in \cite{Aalst2024}, in which specifically the suspension bridge equation is discussed, but only on bounded domains. 

\begin{figure}
    \centering
    \includegraphics[width=0.65\linewidth]{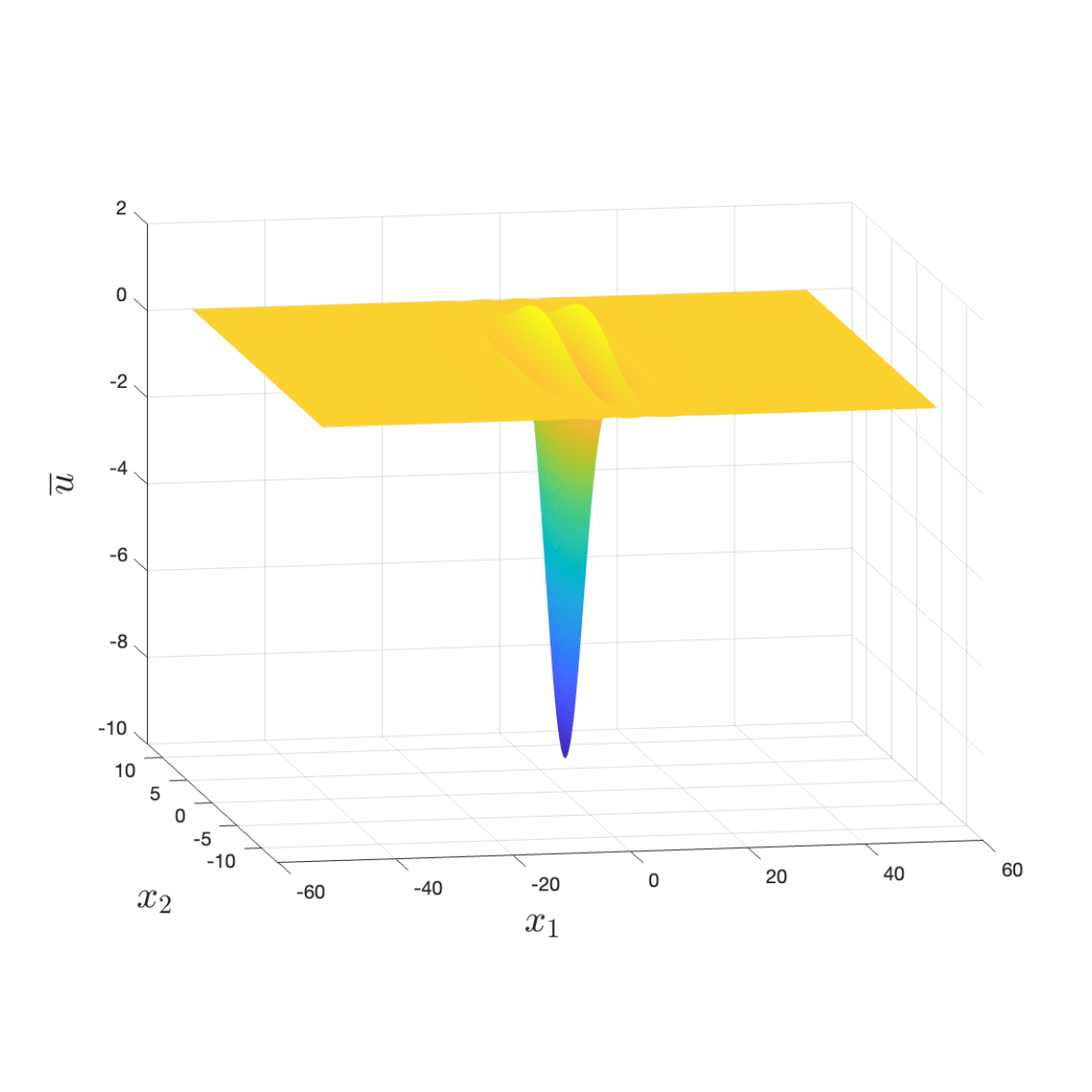}
     \vspace*{-15mm}
    \caption{Visualization of an approximate solution $\bu$ to \eqref{eq : suspension bridge equation} with $c = 1.2$. The approximation is truncated to a finite domain in this plot.}
    \label{fig:proof}
\end{figure}

In this paper, we construct an approximate solution $\bu : \Omega \to \R$ to \eqref{eq : suspension bridge equation} and then prove the existence of a true solution to \eqref{eq : suspension bridge equation} in a vicinity of $\bu$.
An illustration of such an approximate solution $\bu$, that we rigorously verify in Theorem \ref{thm:proof}, is depicted in Figure \ref{fig:proof}. The precise construction of $\bu$ is explained in Section \ref{sec:constructionu0}.  The verification of such an approximate solution is done using a computer-assisted proof (CAP), a method that has gained significant interest in the rigorous verification of numerical simulations for dynamical systems over the past few decades. Early implementations of CAPs to the study of elliptic PDEs are seen in \cite{Nakao} and \cite{Plum}, stemming from 1988 and 1992, respectively. Later, advancements in the applications of CAPs to (elliptic) PDEs were made (e.g. \cite{Cyranka,radii_Day,Oishi,Takayasu,intro_berg2019,Piotr}). Comprehensive discussions on the role of CAPs in the study of PDEs are outlined in \cite{gomez_overview} and \cite{BookOverviewCAP}. 

In particular, the setting in which we work allows us to utilize a Newton-Kantorovich-type theorem. These type of theorems are widely used in the construction of CAPs for PDEs and various other contexts (see for instance \cite{intro_breden,intro_hungria,NS2D}). Specifically, we use a variant of this theorem called the radii polynomial theorem \cite{radii_Day}. To prove existence on a bounded domain using this theorem, we need to compute certain bounds. To compensate for the fact that we are working on an unbounded domain, an additional bound is introduced and computed in this paper. To guarantee the reliability of the computational components of the proof, we employ \textit{Julia} interval arithmetic using the packages \cite{RadiiPolynomial.jl} and \cite{IntervalArithmetic.jl}. 

The solutions obtained in our CAPs will also be analyzed for stability. This part of the work is based on the procedure outlined in \cite{SB_nagatou}, where the authors investigate orbital stability of \textit{one-dimensional} traveling wave solutions to the suspension bridge equation. Their analysis employs energy methods and spectral analysis to establish the stability results, supported by computer-assisted techniques. In two dimensions, to our knowledge, only numerical results concerning stability are available so far, which can be found in \cite{SB_horak}. To rigorously handle the spectral properties of the higher-dimensional equation, we follow the framework proposed in \cite{cadiot2025stabilityanalysislocalizedsolutions}. Using this approach, we are able to conclude that the solution corresponding to the approximation in Figure \ref{fig:proof} is orbitally stable.

The paper is organized as follows. We begin by introducing notation and formulating the zero-finding problem. In Section \ref{sec:existenceproof}, we present the Newton-Kantorovich theorem, which forms the foundation of our existence results. Moreover, we describe the construction of approximate solutions and provide more computational details. Section \ref{sec : computation bounds} presents the bounds required to apply the existence theorem. The set-up for the stability analysis is given in Section \ref{sec:stability}. Finally, the proof of the solution corresponding to the approximation shown in Figure \ref{fig:proof} is given in Section \ref{sec : results}, along with two additional existence results. All three proofs include verifications of orbital (in)stability.

\section{Set-up}
In this section, we introduce notation that will be used in the rest of the paper. Part of the notation is borrowed from \cite{Cadiot2023, Cadiot2024} and adjusted to our setting of the suspension bridge equation. 

\subsection{Zero-finding problem for functions}

Let $\mathbb{L}$ be the linear differential operator defined as 
\begin{align*}
\mathbb{L}=\Delta^2+c^2\partial_{x_1}^2+I
\end{align*}
where $I$ is the identity and $\Delta$ is the Laplace operator. Moreover, let $\mathbb{G}$ be defined as $\mathbb{G}(u)=e^u-u-1$. Then, \eqref{eq : suspension bridge equation} is equivalent to a zero finding problem $\mathbb{F}(u) = 0$ where 
\begin{align}\label{def : zero finding F}
    \mathbb{F}(u) = \mathbb{L}u + \mathbb{G}(u) =  \Delta^2 u + c^2 \partial_{x_1}^2u  + e^{u}  -1.
\end{align}
Note that by defining $\mathbb{G}$ as such we have that $\mathbb{G}(0) = 0$ and $D\mathbb{G}(0) = 0$.
In particular, we look for a solution to \eqref{eq : suspension bridge equation} such that $u(x) \to 0$ as $|x_1| \to \infty$ and $u$ satisfies boundary conditions \eqref{eq : neumann boundary conditions}. Taking advantage of the symmetries $u(x_1,x_2)=u(-x_1,x_2)=u(x_1,-x_2)$ allows us to represent $u$ as a cosine series instead of a general Fourier series. Furthermore, assuming the symmetry in the $x_1$-direction avoids the degeneracy of the Jacobian operator that would otherwise arise from the shift invariance of the problem on the infinite strip. As a result, we choose the following ansatz 
\begin{equation}\label{eq:ansatz}
    u(x_1,x_2) = u_0(x_1)  +  2\sum_{n_2 = 1}^\infty  u_{n_2}(x_1)\cos\left(2\pi \tilde{n}_2x_2\right) ~~ \text{ with } \tilde{n}_2 \bydef \frac{n_2}{2d_2},
\end{equation}
and where $u_{n_2} : \mathbb{R} \to \mathbb{R}$ and $u_{n_2}$ is even.
%
%
%
Corresponding to \eqref{eq:ansatz}, we define $L^2_{e}$ as the following Hilbert space of $L^2$ even functions on the strip $\Omega$:
\begin{align*}
    L^2_{e} \bydef \left\{ u(x_1,x_2) = u_0(x_1) + 2 \sum_{n_2 \in \mathbb{N}} u_{n_2}(x_1) \cos\left(2\pi \tilde{n}_2x_2\right), ~ u_{n_2}(x_1) = u_{n_2}(-x_1) \text{ and } \|u\|_{2}<\infty \right\} 
\end{align*}
where 
\begin{equation*}
    \|u\|_{2} \bydef \left( \|u_0\|_{L^2(\R)}^2 + 2 \sum_{n_2 \in \mathbb{N}}  \|u_{n_2}\|_{L^2(\R)}^2\right)^{\frac{1}{2}}.
\end{equation*}
 Now, given $u \in L^2_{e}$ regular enough, we have that 
\begin{align*}
    \mathbb{L} u (x) = \mathbb{L}_0 u_0(x_1) + 2 \sum_{n_2 \in \mathbb{N}} \mathbb{L}_{n_2} u_{n_2}(x_1) \cos\left(2\pi \tilde{n}_2x_2\right)
\end{align*}
for all $x \in \Omega$, where 
\begin{align*}
    \mathbb{L}_{n_2} \bydef \left(-\frac{n_2^2\pi^2}{d_2^2} + \partial_{x_1}^2\right)^2 + c^2 \partial_{x_1}^2 + I
\end{align*}
for all $n_2 \in \mathbb{N}_0.$ In particular, note that for $c \in (0,\sqrt{2})$, the symbol $l_{n_2}(\xi_1)$ is strictly positive and bounded away from zero for all $\xi_1 \in \mathbb{R}$ (see Lemma \ref{lem : results on l}. Thus, $\mathbb{L}_{n_2} : H^4(\R) \to L^2(\R)$ is bijective with a bounded inverse. Consequently, similarly as what was achieved in \cite{Cadiot2023} and \cite{Cadiot2024}, we can define the Hilbert space $\mathcal{H}$ as follows: 
\begin{align*}
    \mathcal{H} \bydef \{u \in L^2_{e}, ~  \|u\|_\mathcal{H} \bydef \|\mathbb{L}u\|_{2} < \infty\}.
\end{align*}
In addition, we see that
\[
 \|u\|_\mathcal{H}= \left(\|\mathbb{L}_0u_0\|_2^2 + 2\sum_{n_2=1}^\infty \|\mathbb{L}_{n_2}u_{n_2}\|_2^2\right)^{\frac{1}{2}}.
\]
Using the bijectivity of the individual components $\mathbb{L}_{n_2}$, we have that $\mathbb{L} : \mathcal{H} \to L^2_{e}$ is an isometric isomorphism and hence bijective as well. Moreover, we define $l : \R^2 \to \R$ as the symbol of $\mathbb{L}$ and $l_{n_2} : \R \to \R$ as the one of $\mathbb{L}_{n_2}$ for all $n_2 \in \mathbb{N}_0$. Specifically, we have that 
\begin{equation}\label{eq:defl}
    l(\xi) \bydef |2\pi \xi|^4 - c^2 (2\pi \xi_1)^2 + 1 ~~ \text{ and } ~~ l_{n_2}(\xi_1) \bydef (2\pi)^4(  \xi_1^2 +  \tilde{n}_2^2 )^2 -  c^2 (2\pi \xi_1)^2 + 1
\end{equation}
where we recall that $\tilde{n}_2 \bydef \frac{n_2}{2d_2}$.

\begin{remark}
Note that one could impose different boundary conditions in the $x_2$-direction. For instance, a variational formulation of \eqref{eq : suspension bridge equation} leads to natural ``free'' boundary conditions which read 
$\Delta u(\cdot,\pm d_2)=\partial_{x_2}\Delta u(\cdot,\pm d_2)=0.$  
However, these conditions do not allow to use Fourier series in the $x_2$-direction (in contrast to \eqref{eq:ansatz}), which leads to technical and computational challenges in the current approach. This hurdle might be addressed using a different expansion in the $x_2$ direction (such as a polynomial expansion), and would require an adaptation of the presented analysis. 
\end{remark}

Now, it remains to prove that $\mathbb{G} : \mathcal{H} \to L^2_{e}$ is well defined in order to formulate \eqref{eq : suspension bridge equation} on $\mathcal{H}$. For that purpose, we use the following result:
\begin{lemma}\label{lem : banach algebra}
Let $\kappa_1,\kappa_2$ be defined as follows:
\begin{equation}\label{def : definition of kappa1 kappa2}
    \kappa_1 \bydef \frac{1}{1-\frac{c^4}{4}} ~ \text{ and } ~ \kappa_2 \bydef \left( \left\|\frac{1}{l_0}\right\|_{L^2(\R)}^2 + 2 \sum_{n_2 \in \mathbb{N} } \left\|\frac{1}{l_{n_2}}\right\|_{L^2(\R)}^2\right)^{\frac{1}{2}}.
\end{equation}
Then, for all $u, v \in \mathcal{H}$, we have that
\begin{align*}
    \|(e^{u}-1)v\|_{2} \leq \kappa_1 (e^{\kappa_2\|u\|_\mathcal{H}}-1)\|v\|_{\mathcal{H}}.
\end{align*}
\end{lemma}
\begin{proof}
First, we notice that 
\begin{align*}
     \|(e^{u}-1)v\|_{2} \leq \|e^{u}-1\|_\infty \|v\|_{2}.
\end{align*}
Then, we have
\begin{align*}
    \|e^{u}-1\|_\infty \leq e^{\|u\|_\infty}-1.
\end{align*}
 It remains to prove that $\|u\|_\infty \leq \kappa_2 \|u\|_\mathcal{H}$ and $\|v\|_{2} \leq \kappa_1 \|v\|_\mathcal{H}$. Using the proof of Lemma 2.1 in \cite{Cadiot2023}, we obtain that 
 \begin{align*}
     \|u\|_\infty \leq \left( \left\|\frac{1}{l_0}\right\|_{L^2(\R)}^2 + 2 \sum_{n_2 \in \mathbb{N} } \left\|\frac{1}{l_{n_2}}\right\|_{L^2(\R)}^2\right)^{\frac{1}{2}} \|u\|_\mathcal{H} \text{ and } \|v\|_{2} \leq \max_{n_2\in \mathbb{N}_0}\sup_{\xi_1 \in \R} \left| \frac{1}{l_{n_2}(\xi_1)}\right| \|v\|_\mathcal{H}.
 \end{align*}
We conclude the proof using \eqref{eq : minimum of 1 over l} from Lemma \ref{lem : results on l}.
\end{proof}

\begin{remark}
   In practice, for the computation of $\kappa_2$, we use Riemann summations to bound the summation up to finitely many $n_2$. In fact, this can be achieved using rigorous numerics (cf. \cite{julia_our_code}). To estimate the tail of the summation, we use \eqref{eq:tail_lnsum} in Lemma \ref{lem : results on l}.
\end{remark}

Using the embedding of $\mathcal{H}$ into $L^\infty(\Omega)$ established in Lemma \ref{lem : banach algebra} and the analyticity of the exponential function, it follows that $\mathbb{G} : \mathcal{H} \to L^2_{e}$ is a Fréchet differentiable mapping.
Accordingly, we investigate the solution of the following zero-finding problem \begin{align}\label{eq : zero finding on H}
    \mathbb{F}(u)=0 ~~ \text{ for } u \in \mathcal{H},
\end{align} 
where $\mathbb{F}: \mathcal{H} \rightarrow L^2_{e}$.

\subsection{Fourier series representation}\label{sec:fourierrep}

As described in \cite{Cadiot2023}, existence proofs of localized solutions can be achieved using approximations of the true solution in Fourier space. Following such an approach, we introduce notation corresponding to Fourier series, related to the ones introduced in the previous section.
Let $d_1>0$ and let $\Omega_0 \bydef (-d_1,d_1) \times (-d_2,d_2)$ be the bounded domain on which our approximate solution is constructed.
 We let $\mathbb{N}_0^2=(\mathbb{N} \cup \{0\})^2$  and we define
\begin{align}\label{def : definition of the alpha coefficients}
    \alpha_n=\alpha_{n_1,n_2}:=\begin{cases}
    1 & \text{ for } n_1=n_2=0,\\
      2 & \text{ for } n_1=0,n_2>0,\\
        2 &\text{ for } n_1>0,n_2=0,\\
          4 &\text{ for } n_1>0,n_2>0
    \end{cases}
\end{align}
for all $n \in \mathbb{N}_0^2$.
The coefficients $\alpha_n$ appear when switching between the cosine series representation and the exponential Fourier series one. Let $\ell^p$ denote the Lebesgue space for sequences defined as
\[
    \ell^p \bydef \left\{U = (U_n)_{n \in \mathbb{N}_0^2}: ~ \|U\|_p \bydef \left( \sum_{n \in \mathbb{N}_0^2} \alpha_n|U_n|^p\right)^\frac{1}{p} < \infty \right\}.
\]
Then, operators $\mathbb{L}$ and $\mathbb{G}$ have a Fourier coefficients representation, when being considered on $\Omega_0$ with periodicity in the $x_1$-direction. Indeed, defining $\tilde{n} = (\tilde{n}_1,\tilde{n}_2) = \left(\frac{n_1}{2d_1},\frac{n_2}{2d_2}\right)$, we introduce $L$, the Fourier coefficients representation of $\mathbb{L}$, as 
\[
(LU)_n = l(\tilde{n}) U_n \text{ for all } n \in \mathbb{N}_0^2
\]
and all $U \in \ell^2$. Then, we define $U*V$ as the discrete convolution for sequences indexed on $\mathbb{N}_0^2$, given as 
\begin{align*}
    (U \ast V)_n=\sum_{m \in \mathbb{Z}^2}U_{|m|}V_{|n-m|}
\end{align*}
for $n \in \mathbb{N}_0^2$. Note that Young's convolution inequality holds true for this convolution product, i.e., for $U \in \ell^2$ and $V \in \ell^1$, we have that
\begin{align}\label{eq:young}
    \|U * V\|_2 \leq \|U\|_2 \|V\|_1.
\end{align}
Moreover, given $U \in \ell^1$, we define $e^{U}$ as 
\begin{equation}\label{eq:powerseries}
    e^{U} \bydef \sum_{k=0}^\infty \frac{U^k}{k!},
\end{equation}
where $U^k = \overbrace{U*U*\dots*U}^{k \text{ times}}$. Consequently, we define $G: \ell^1 \to \ell^2$ as the equivalent of $\mathbb{G}$ as
\[
G(U) \bydef e^U - U - e_0,
\]
where $(e_0)_n = 1$ if $n=(0,0)$  and $(e_0)_n = 0$ for all $n \in \mathbb{N}^2 \setminus \{0,0\}$. Since $\ell^1$ is a Banach algebra under the discrete convolution product and because of \eqref{eq:powerseries}, $e^U \in \ell^1$, and thus in $\ell^2$, for any $U \in \ell^1$, which ensures the operator is well-defined on the sequence space. Finally, we define the zero finding counterpart $F: \ell^1 \to \ell^2$ on $\Omega_0$ as 
\[
F(U) \bydef LU + G(U).
\]
Now, as in \cite{Cadiot2023} and \cite{Cadiot2024}, we introduce operators which allow us to switch from $L^2_e$ to $\ell^2$ and vice versa. Therefore, we define $\gamma:L^2_e\rightarrow \ell^2$ and $\gamma^\dagger: \ell^2 \rightarrow L^2_e$ such that
\begin{align*}
    \left(\gamma(u)\right)_n&=\frac{1}{|\Omega_0|}\int_{\Omega_0}u(x_1,x_2)e^{-2\pi i \tilde{n} \cdot x}dx \quad \text{ for } n\in \mathbb{N}_0^2,\\
    \gamma^\dagger(U)(x)&=\mathbb{1}_{\Omega_0}(x)\sum_{n\in \mathbb{N}_0^2} \alpha_n U_n\cos\left(2\pi \tilde{n}_1 x_1\right)\cos\left(2\pi \tilde{n}_2 x_2\right),
\end{align*}
where $\mathbb{1}_{\Omega_0}$ denotes the characteristic function on $\Omega_0$.

Now, given $X$ a Banach space, we denote $\mathcal{B}(X)$ as the set of bounded linear operators on $X$.
Then, we define $\Gamma : \mathcal{B}(L^2_e) \to \mathcal{B}(\ell^2)$ and $\Gamma^\dagger : \mathcal{B}(\ell^2) \to \mathcal{B}(L^2_e)$ as
\begin{align}\label{def : gamma and gamma dagger}
    \Gamma(\mathbb{H}) = \gamma \mathbb{H}\gamma^\dagger ~  \text{ and } ~ 
    \Gamma^\dagger(H) = \gamma^\dagger H \gamma
\end{align}
for all bounded linear operators $\mathbb{H} \in \mathcal{B}(L^2_e)$ and $H \in \mathcal{B}(\ell^2)$. The map $\Gamma^\dagger$ will be useful when constructing operators on $\mathcal{B}(L^2_e)$ thanks to operators on Fourier coefficients (that is, operators on $\mathcal{B}(\ell^2)$).
Furthermore, we introduce the norm $\|\cdot\|_{2,\mathcal{H}}$ to denote the operator norm for bounded linear operators $L^2_e \rightarrow \mathcal{H}$. 

Finally, given $u \in L^\infty, U \in \ell^1$, we introduce notation to denote the linear multiplication operator $\mathbb{M}_u$ related to the function $u$ and the linear discrete convolution operator $\mathbb{M}_U$ related to the sequence $U$ as 
\begin{align}\label{def : multiplication operator}
    \mathbb{M}_u: L^2_e \to L^2_e : {v}\mapsto \mathbb{M}_u v=uv ~ \text{ and } ~
    \mathbb{M}_U: \ell^2 \to \ell^2 : V\mapsto \mathbb{M}_U V=U \ast V.
\end{align}

\section{Constructive proof of existence}\label{sec:existenceproof}
In this section, we present our approach for proving constructively the existence of solutions to \eqref{eq : suspension bridge equation}. Our approach is based on a Newton-Kantorovich theorem. More specifically, we construct a fixed point operator for which fixed points correspond to zeros of $\mathbb{F}$.

First, we construct an approximate solution $\bar{u} \in \mathcal{H}$ thanks to its Fourier coefficients representation. The technical construction of $\bar{u}$ is described in Section \ref{sec:constructionu0}. Given $\bar{u}$, we then construct an injective $\mathbb{A} : L^2_e \to \mathcal{H}$ as an approximate inverse for the Fréchet derivative $D\mathbb{F}(\bar{u})$. Then, our goal is to determine a radius $r>0$ such that $\mathbb{T}$, given as 
\[
\mathbb{T}(u) \bydef u - \mathbb{A} \mathbb{F}(u),
\]
is well defined on $\overline{B_r(\bar{u})}  \to \overline{B_r(\bar{u})}$ and contracting (where $B_r(\bar{u})$ is the open ball of $\mathcal{H}$ centered at $\bar{u}$ and of radius $r$). The Banach fixed point theorem then provides the existence of a unique zero of $\tilde{u} \in \mathcal{H}$ of $\mathbb{F}$ in $\overline{B_r(\bar{u})}.$ 

In order to verify the existence of such a radius $r$, we make use of the following result. Its proof is for instance given in \cite{Cadiot2023}.
\begin{theorem}\label{thm:existence}
    Let $\mathbb{F}:\mathcal{H} \to L^2_e$ be given as in \eqref{def : zero finding F}   and let $\mathbb{A}: L^2_e \rightarrow \mathcal{H}$ be a bounded linear operator. Assume that $\mathcal{Y}_0, \mathcal{Z}_1 \in (0,\infty)$ and let $\mathcal{Z}_2: (0,\infty) \to [0,\infty)$ satisfy 
    \begin{align}
        \|\mathbb{A}\mathbb{F}(\bu)\|_\mathcal{H} &\leq \mathcal{Y}_0, \label{eq:Y0}\\
        \|I-\mathbb{A}D\mathbb{F}(\bu)\|_\mathcal{H}& \leq \mathcal{Z}_1,\label{eq:Z1}\\
        \|\mathbb{A}(D\mathbb{F}(\bu+h)-D\mathbb{F}(\bu))\|_\mathcal{H} &\leq \mathcal{Z}_2(r)r \quad \text{ for all } h\in \overline{B_r(0)} \subset \mathcal{H} \text{ and all } r>0. \label{eq:Z2bound}
    \end{align}
    If there exists $r>0$ such that 
    \begin{align}\label{eq : condition radii localized}
        \frac{1}{2}\mathcal{Z}_2(r)r^2-(1-\mathcal{Z}_1)r+\mathcal{Y}_0<0 \text{ and } \mathcal{Z}_1+\mathcal{Z}_2(r)r<1,
    \end{align}
    then there exists a unique $\tilde{u}\in \overline{B_r(\bu)} \subset \mathcal{H}$ for which $\mathbb{F}(\tilde{u})=0$.
\end{theorem}
Observe that we have transformed our problem into computing explicit values for the bounds $\mathcal{Y}_0, \mathcal{Z}_1, \mathcal{Z}_2$ of the previous theorem. Then, a value for the radius $r>0$ is obtained thanks to \eqref{eq : condition radii localized}.

The core of our approach becomes determining $\bar{u}\in \mathcal{H}$ and $\mathbb{A} : L^2_e \to \mathcal{H}$ such that the bounds of Theorem \ref{thm:existence} can be computed explicitly. We describe the construction of these objects in the next sections.
\subsection{Construction of an approximate solution}\label{sec:constructionu0}
First, we demonstrate how to construct the approximate solution $\bar{u} \in \mathcal{H}$. Our approach is based on that presented in \cite{Cadiot_2025_GS, Cadiot2023, Cadiot2024}. We construct $\bar{u}$ using a cosine-cosine series representation on $\om = (-d_1,d_1) \times (-d_2,d_2)$ as 
\begin{align}\label{def : ansatz for bar u}
  \bar{u}(x) \bydef  \mathbb{1}_{\Omega_0}(x)\sum_{n\in \mathbb{N}_0^2} \alpha_n \bar{U}_n\cos\left(2\pi \tilde{n}_1 x_1\right)\cos\left(2\pi \tilde{n}_2 x_2\right)
\end{align}
where $\bar{U} = (\bar{U}_n)_{n \in \mathbb{N}_0^2}$ is the Fourier coefficients representation of $\bar{u}$. Note that by construction supp$(\bar{u}) \subset \overline{\om}$. The sequence $\bbu$ is determined numerically and contains a finite number of non-zero coefficients. In order to describe our numerical truncations, we define the following projection operators : given $\mathcal{N} = (\mathcal{N}_1, \mathcal{N}_2) \in \mathbb{N}^2$
\begin{align}\label{def : projection operators}
    \left( \pi^\mathcal{N}(U)\right)_n = \begin{cases}
 U_n & \text{ for } n \in I_{\mathcal{N}},\\ 0 & \text{ for }n \notin I_{\mathcal{N}}
    \end{cases}
~ \text{ and } ~ 
    \left( \pi_\mathcal{N}(U)\right)_n = \begin{cases}
 0 &  \text{ for } n \in I_{\mathcal{N}},\\ U_n & \text{ for } n \notin I_{\mathcal{N}}.
    \end{cases}
\end{align}
where $ I_\mathcal{N} \bydef \{n \in \mathbb{N}_0^2 : 0\leq n_1\leq \mathcal{N}_1,0\leq n_2 \leq \mathcal{N}_2\}$.

Let us fix $N^0 = (N^0_1,N^0_2) \in \mathbb{N}^2$ to be the numerical truncation size of $\bbu$. Then,
using \eqref{def : projection operators}, we assume that $\bbu = \pi^{N^0} \bbu$. That is, $\bbu$ has a vector representation of size $(N_1^0+1)(N_2^0+1).$

Now, note that by construction $\bu \in L^2_e$, but we do not necessarily have $\bu \in \mathcal{H}.$ Indeed, $\bu$ might not be smooth at $x_1 = \pm d_1$. By periodicity of $\bu$ on $\om$, smoothness at $x_1 = \pm d_1$ is equivalent to smoothness only at $d_1$. Now, in order to obtain the required regularity for $\mathcal{H}$, we need 
\begin{align*}
    \partial_{x_1}^k \bu(d_1,x_2) = 0 \text{ for all } x_2 \in (-d_2,d_2) \text{ and all } k \in \{0,1,2,3\}. 
\end{align*}
Since $\bu$ has a cosine-cosine representation, we readily have $\partial_{x_1}^k \bu(d_1,x_2) = 0$ for all  $x_2 \in (-d_2,d_2)$  and all  $k \in \{1,3\}$. Consequently, it remains to ensure that 
\begin{align}\label{eq : trace zero condition function}
    \partial_{x_1}^k \bu(d_1,x_2) = 0 \text{ for all } x_2 \in (-d_2,d_2) \text{ and all } k \in \{0,2\}. 
\end{align}
As $\bu$ is represented by $\bbu$ on $\om$, \eqref{eq : trace zero condition function} can be translated to equations on $\bbu$ :
\begin{align*}
    \sum_{n_1 = -N^0_1}^{N^0_1} \alpha_{n_1,n_2} (-1)^{n_1} (2\pi \tilde{n}_1)^k \bbu_{n_1, n_2} = 0 \text{ for all } n_2 \in -N^0_2, \dots, N^0_2 \text{ and all } k \in \{0,2\},
\end{align*}
where we have used that $\bbu = \pi^N \bbu$. Note that this is equivalent to having $\bbu \in \mathrm{Ker}(\mathcal{T})$, where 
\[
\mathcal{T}\bbu = \begin{pmatrix}
   \mathcal{T}_0\bbu\\
   \mathcal{T}_2\bbu
\end{pmatrix}
 \bydef 
 \begin{pmatrix}
    \left(\sum_{n_1 = -N^0_1}^{N^0_1} \alpha_{n_1,n_2} (-1)^{n_1}  \bbu_{n_1, n_2}\right)_{n_2 \in -N^0_2,\dots, N^0_2}\\
    \left(\sum_{n_1 = -N^0_1}^{N^0_1} \alpha_{n_1,n_2} (-1)^{n_1} (2\pi \tilde{n}_1)^2 \bbu_{n_1, n_2}\right)_{n_2 \in -N^0_2,\dots, N^0_2}
\end{pmatrix}.
\]
Consequently, $\bbu$ needs to be projected into the kernel of the matrix $\mathcal{T}$ of size $(2N^0_1+1)(2N^0_2+1)$. Such an operation has been described in \cite{Cadiot2023, Cadiot2024} and is implemented in \cite{julia_our_code}.

Overall, we ensure that $\bbu \in \mathrm{Ker}(\mathcal{T})$ and obtain that
\begin{equation}\label{def : approx solution}
    \bu \bydef \gamma^\dagger\left(\bbu\right) \in \mathcal{H} ~ \text{ and } ~ \bbu = \pi^N \bbu.
\end{equation}
By construction, we have supp$(\bu) \subset \om$.

Now that $\bu$ has been constructed, one of the major difficulties in the treatment of \eqref{eq : suspension bridge equation} is the evaluation of the nonpolynomial nonlinearity $e^{\bu}$. Indeed, in order to evaluate the bound $\mathcal{Y}_0$ \eqref{eq:Y0}, we need to be able to compute rigorously such a quantity. This is the objective of the next section.

\subsection{Computing the nonlinearity}\label{sec:nonlin}

The existence proof requires the computation of the nonlinear part ${G}(U)=e^{U}-U-e_0$. However, the Fourier coefficients corresponding to the exponential function cannot be determined exactly, as there are infinitely many nonzero Fourier coefficients. Moreover, even the computation of a finite number of Fourier coefficients cannot be performed exactly due to the nonpolynomial nature of the nonlinearity. For small indices, we approximate these coefficients using discrete Fourier transforms and establish an upper bound on the associated error, referred to as the aliasing error. For larger indices, we use a general bound, which is only suitable for the tail part because of the decay of the bound. To compute the Fourier coefficients of ${G}(\bbu)$, we first determine the Fourier coefficients $e^{\bbu}$ corresponding to the function $e^{\bu}$. Following the procedure outlined in Section 3 of \cite{Aalst2024}, we obtain the expression: 
\begin{align}\label{eq:alias}
   \left( e^{\bbu} \right)_n \in \begin{cases}
        {(e^{\bbu})}^{\mathrm{FFT}}_n+{C}[-\varepsilon_n,\varepsilon_n] & \text{ if } 0 \leq n_1\leq N^0_1, 0 \leq  n_2 \leq N^0_2,\\
        [-\frac{{C}}{\nu_1^{|n_1|}\nu_2^{|n_2|}},\frac{{C}}{\nu_1^{|n_1|}\nu_2^{|n_2|}}] & \text{ otherwise} 
    \end{cases}
\end{align}
 where ${(e^{\bbu})}^{\mathrm{FFT}}$ is the numerical approximation of $e^{\bbu}$ calculated using discrete Fourier transforms, $C>0$, $\nu_1,\nu_2>1$ and $\varepsilon_n>0$ for all $n$. The procedure of computing $C$ can be found in \cite{Aalst2024}. The parameters $\nu_1,\nu_2$ are related to the domain of analyticity, i.e., $\bbu$ is analytic on a strip $\{z \in \mathbb{C}^2 : \mathrm{Im}(z_1)<\rho_1, \mathrm{Im}(z_2)<\rho_2\}$ and $\nu_1 = e^{\rho_1-\overline{\varepsilon}_1}$, $\nu_2 = e^{\rho_2-\overline{\varepsilon}_2}$ for small $\overline{\varepsilon}_1,\overline{\varepsilon}_2>0$ (cf. Section 3.1 in \cite{Aalst2024}). Note that ${G}(\bbu)$ is analytic for $\bbu$ analytic.

Moreover, let $N^\mathrm{FFT}$ denote how many Fourier modes are considered when doing Fourier transformations such that $N^\mathrm{FFT}_1$ and $N^\mathrm{FFT}_2$ are powers of $2$ and $N^\mathrm{FFT} \gg N^0$. Then we have \begin{align*}
\varepsilon_n:=\frac{2\nu_1^{|n_1|}\nu_2^{|n_2|}\left(\nu_1^{-2N_1^\mathrm{FFT}}+\nu_2^{-2N_2^\mathrm{FFT}}\right)}{\left(1-\nu_1^{-2N_1^\mathrm{FFT}}\right)\left(1-\nu_2^{-2N_2^\mathrm{FFT}}\right)}.  
\end{align*}

\begin{remark}
Note that the support of $\bu$ is on $\overline{\Omega_0}$, which is a finite domain. As a consequence, $\mathbb{G}(\bu)$ also has its support on $\overline{\Omega_0}$. Therefore, the actual representation of the nonlinear term is on the infinite strip, even though Fourier transformations are carried out on finite domains only.
\end{remark}

\subsection{Approximate inverse}

In order to be able to apply Theorem \ref{thm:existence}, it remains to construct an approximate inverse $\mathbb{A}$ for $D\mathbb{F}(\bu)$. For this purpose, we recall the construction established in \cite{Cadiot2023}. Let $N = (N_1, N_2) \in \mathbb{N}^2$ be the numerical truncation size of our operators. In practice, we choose $N_1 \leq N^0_1$ and $N_2 \leq N^0_2$, where $N^0$ is the size of the numerical truncation of $\bar{U}$.

Numerically, we construct a ``matrix'' $B^N$ approximating the inverse of $\pi^N DF(\bbu)L^{-1}\pi^N$. In particular, $B^N$ is finite-dimensional in the sense that $B^N = \pi^N B^N \pi^N$. Then, using \eqref{def : gamma and gamma dagger}, we define $\mathbb{B} : L^2_e \to L^2_e$ as 
\begin{align*}
    \mathbb{B} \bydef \mathbb{1}_{\Omega \backslash \Omega_0}+\Gamma^\dagger(\pi_N+B^N).
\end{align*}
Here, $\mathbb{1}_{\Omega \backslash \Omega_0}$ has to be understood as a multiplication operator by the characteristic function $\mathbb{1}_{\Omega \backslash \Omega_0}$. In particular, $\mathbb{B}$ approximates the inverse of $D\mathbb{F}(\bu)\mathbb{L}^{-1}.$ Finally, we define $\mathbb{A}$ as 
\begin{align}\label{def : approx inverse A}
    \mathbb{A} \bydef \mathbb{L}^{-1}\mathbb{B} = \mathbb{L}^{-1}\left(\mathbb{1}_{\Omega \backslash \Omega_0}+\Gamma^\dagger(\pi_N+B^N)\right) : L^2_e \to \mathcal{H}.
\end{align}
Note that $\mathbb{A}$ is well defined as $\mathbb{L} : \mathcal{H} \to L^2_e$ is an isometric isomorphism. Computing the bound $\mathcal{Z}_1$ defined in \eqref{eq:Z1}, we will show that \eqref{def : approx inverse A} is an accurate choice for an approximate inverse. In particular, it is well suited for our analysis since it allows us to compute the bounds of Theorem \ref{thm:existence} explicitly. This will be the goal of the next section. We can already recall a result from \cite{Cadiot2023} providing the operator norm of $\mathbb{A}$ thanks to the norm of the matrix $B^N$:
\begin{align}\label{eq:Anorm}
\|\mathbb{A}\|_{2,\mathcal{H}}=\|\mathbb{B}\|_{2} = \max\{1,\|B^N\|_{2}\}.
\end{align}

\section{Computation of the bounds}\label{sec : computation bounds}

In this section, we justify our choices for the approximate objects constructed above by providing explicit formulas for the bounds of Theorem \ref{thm:existence}. Before stating the first bound, we recall that $N^0$  corresponds to the truncation dimension for sequences and $N$ for operators.

\begin{lemma}[\bf Bound \boldmath$\mathcal{Y}_{0}$\unboldmath]
  Let $\mathcal{Y}_0$ be satisfying
    \begin{align*}
      \mathcal{Y}_0\geq \sqrt{2d_1}&\Bigg( \|B^NF(\bbu)\|_{2}^2+\|(\pi^{N^0}-\pi^N)(L\bbu+G(\bbu))\|_{2}^2\\
     &+ C^2\frac{2\nu_1^{-2N^0_1-2}(1+\nu_2^{-2})+2\nu_2^{-2N^0_2-2}(1+\nu_1^{-2})-4\nu_1^{-2N_1^0-2}\nu_2^{-2N_2^0-2}}{(1-\nu_1^{-2})(1-\nu_2^{-2})}\Bigg)^{\frac{1}{2}}.
  \end{align*}
  Then $\mathcal{Y}_0$ satisfies \eqref{eq:Y0}.
\end{lemma}
\begin{proof}
As $\mathrm{supp}(\bu)\subset \overline{\Omega_0}$, we have that $\mathrm{supp}(\mathbb{G}(\bu))\subset \overline{\Omega_0}$ since $e^u-u-1=\sum_{k=2}^\infty \frac{u^k}{k!}$. Hence, we can apply Parseval's identity and  follow the proof of Lemma 4.11 in \cite{Cadiot2023} to obtain
\begin{align*}
    \|\mathbb{A}\mathbb{F}(\bu)\|_\mathcal{H}=\|\mathbb{L}\mathbb{A}\mathbb{F}(\bu)\|_{2}=\|\mathbb{B}\mathbb{F}(\bu)\|_{2}=\sqrt{2d_1}\|BF(\bbu)\|_{2}\end{align*}
    where
\begin{align}\label{eq:BF}
    \|BF(\bbu)\|_{2}^2=\|B^NF(\bbu)\|_{2}^2+\|(\pi^{N^0}-\pi^N)L\bbu+\pi_NG(\bbu)\|_{2}^2.
\end{align}
Note that the nonlinear part of $F(\bbu)$ in $\|B^NF(\bbu)\|_{2}^2$ can be calculated and bounded using \eqref{eq:alias}. 

For the second term in the right-hand side of \eqref{eq:BF} we write 
\begin{align*}
    \|(\pi^{N^0}-\pi^N)L\bbu+\pi_NG(\bbu)\|_{2}^2&\leq \|(\pi^{N^0}-\pi^N)(L\bbu+G(\bbu))\|_{2}^2\\
    &+\sum_{ n \notin I_{N^0}}  \alpha_n \frac{C^2}{\nu_1^{2|n_1|}\nu_2^{2|n_2|}}.
\end{align*}
The last summation can be calculated using a geometric series argument as was done in Lemma 3.4 of \cite{Aalst2024}. Hence,
\begin{align*}
    \sum_{ n \notin I_{N^0}} &\alpha_n \frac{C^2}{\nu_1^{2|n_1|}\nu_2^{2|n_2|}}\\
    &=C^2\frac{2\nu_1^{-2N^0_1-2}(1+\nu_2^{-2})+2\nu_2^{-2N^0_2-2}(1+\nu_1^{-2})-4\nu_1^{-2N_1^0-2}\nu_2^{-2N_2^0-2}}{(1-\nu_1^{-2})(1-\nu_2^{-2})}.
\end{align*}
\end{proof}

\begin{remark}
   Note that for the second term in the $\mathcal{Y}_0$-bound, we can use \eqref{eq:alias} to state the interval that contains this term as follows:
  $$\|(\pi^{N^0}-\pi^N)(L\bbu+G(\bbu))\|_{2}^2\in \sum_{\substack{n \in I_{N^0}\\ n \notin I_{N}}} \alpha_n \left( l\left(\frac{n}{2d}\right)\bbu_n+(e^{\bbu})_n^\mathrm{FFT}+C[-\varepsilon_n,\varepsilon_n]\right)^2.$$
  The right-hand side is what is used in the (interval arithmetic) computation of the $\mathcal{Y}_0$-bound in \cite{julia_our_code}.
\end{remark}

We now turn to the computation of $\mathcal{Z}_2(r)$.

\begin{lemma}[\bf Bound \boldmath$\mathcal{Z}_{2}$\unboldmath]\label{lem : Z2 bound}
Let $r>0$ and let $\mathcal{Z}_{2}(r)$ be satisfying 
   \begin{align*}
   \mathcal{Z}_2(r)\geq \kappa_1\frac{e^{\kappa_2 r}-1}{r}\max\left\{\| {B}e^{\mathbb{M}_{\bbu}}\|_2, 1\right\},
   \end{align*}
      where $\kappa_1$ and $\kappa_2$ are given in Lemma \ref{lem : banach algebra}. Then $\mathcal{Z}_2(r)$ satisfies \eqref{eq:Z2bound}.
\end{lemma}
\begin{proof}
    Let $h\in \overline{B_r(0)}$. Since $\|u\|_{\mathcal{H}}=\|\mathbb{L}u\|_{{2}}$ and $\mathbb{A} = \mathbb{L}^{-1} \mathbb{B}$ by definition, we have that 
    \begin{align*}
        \|\mathbb{A}(D\mathbb{F}(\bu+h)-D\mathbb{F}(\bu))\|_\mathcal{H}&=\|\mathbb{L}\mathbb{A}(D\mathbb{F}(\bu+h)-D\mathbb{F}(\bu))\|_{\mathcal{H},{2}}\\
        &=\|\mathbb{B}(e^{\mathbb{M}_{\bu+h}}-e^{\mathbb{M}_{\bu}})\|_{\mathcal{H},{2}}.
    \end{align*}
Then, we have
\begin{align*}
   \|\mathbb{B}(e^{\mathbb{M}_{\bu+h}}-e^{\mathbb{M}_{\bu}})\|_{\mathcal{H},{2}}&\leq \underset{\|w\|_\mathcal{H}=1}{\sup}
\|\mathbb{B}e^{\mathbb{M}_{\bu}}\|_{2}\|(e^{h}-1)w\|_{L^2_e}
\end{align*}
where, using \eqref{eq:Anorm}, we also get
\begin{align*}
    \|\mathbb{B}e^{\mathbb{M}_{\bu}}\|_{2} \leq  \max\left\{\| {B}e^{\mathbb{M}_{\bbu}}\|_2, 1\right\}.
\end{align*}
We conclude the proof using Lemma \ref{lem : banach algebra}, from which we deduce that
\begin{align*}
    \|(e^{h}-1)w\|_{L^2_e} \leq \kappa_1 (e^{\kappa_2 r}-1).
\end{align*}
\end{proof}

For the computation of the bound $\mathcal{Z}_1$, we first introduce some extra notation for simplicity.
Let $\overline{v} \bydef e^{\bu} - 1$ and $\bar{V} \bydef \gamma(\bar{v})$. Doing so, notice that $D\mathbb{G}(\bu)$ and $DG(\bbu)$ can be expressed as 
\[
D\mathbb{G}(\bu) = \mathbb{M}_{\bar{v}} \text{ and } DG(\bbu) = M_{\bar{V}},
\]
where we have used the notation in \eqref{def : multiplication operator}. Finally, note that $\bar{V}$ is an infinite number of Fourier coefficients, as explained in Section \ref{sec:nonlin}. In order to exhibit finite-dimensional computations, we define $\bar{V}^N$ and $\bar{v}^N$ as 
\[
\bar{V}^N \bydef \pi^N \bar{V} \text{ and } \bar{v}^N \bydef \gamma^\dagger(\bar{V}^N).
\]
That is, $\bar{V}^N$ correspond to the $N$ first modes of $\bar{V}$. In fact, the difference $\bar{V} - \bar{V}^N$ can be controlled explicitly as a result of the analysis of Section \ref{sec:nonlin}. Then, we recall Lemma 3.6 from \cite{Cadiot2024}, providing the decomposition of $\mathcal{Z}_1$ as a Fourier coefficients computation $Z_1$, and a remainder $\mathcal{Z}_u$ that corresponds to the unboundedness of the domain $\Omega.$
\begin{lemma}[\bf Bound \boldmath$\mathcal{Z}_{1}$\unboldmath]\label{lem : computation of Z1}
Let $\mathcal{Z}_{u}$ and $Z_1$ be bounds satisfying 
\begin{equation} \label{def : Z1 periodic and Zu}
\begin{aligned}
\|\left(\mathbb{L}^{-1}-\Gamma^\dagger(L^{-1})\right) \mathbb{M}_{\overline{v}^N} \mathbb{B}^* \|_2 & \le \mathcal{Z}_{u} ,\\
\|I - (B^N + \pi_N )(I + M_{\bar{V}^N} L^{-1})\|_{2} & \le Z_1.
\end{aligned}
\end{equation}
Then defining $\mathcal{Z}_1$ as 
\begin{equation}\label{def : definition of Z1 theoretical}
   \mathcal{Z}_1 \bydef Z_1 +  \left( \mathcal{Z}_{u} +  \frac{\max\{1, \|B^N\|_{2}\}}{1-\frac{c^4}{4}}\|\bar{V}-\bar{V}^N\|_1 \right),
\end{equation}
 we get that \eqref{eq:Z1} holds, that is $ \|I - \mathbb{A}D\mathbb{F}(\bu)\|_\mathcal{H} \leq \mathcal{Z}_1.$
\end{lemma}
\begin{proof}
    The proof is given in \cite{Cadiot2024}. Note that the difference in the formulas lies in the fact that $\sup_{\xi \in \R^2} \frac{1}{l(\xi)} = \frac{1}{1 - \frac{c^4}{4}}$ in the case of \eqref{eq : suspension bridge equation}.
\end{proof}

\begin{remark}
    In our implementation, the operator $\mathbb{L}$ acts as a diagonal operator. Its inverse $\mathbb{L}^{-1}$ is computed by dividing each Fourier coefficient by the corresponding value of the symbol $l_{n_2}(\xi_1)$. This operation is well-defined because the symbol is strictly positive for $c\in(0,\sqrt{2})$.
\end{remark}

We recall the computation of the bound $Z_1$, which is given in Lemma 3.9 from \cite{Cadiot2024}. 

\begin{lemma}\label{lem : usual term periodic Z1}
Let $Z^N_1$ and $Z_1$ be such that
\begin{equation}\label{def : upper bound periodic Z1}
\begin{aligned}
     \left(\|\pi^N - B^N(I_d + M_{\bar{V}^N} L^{-1})\pi^{2N}\|_2^2 + \|(\pi^{2N}-\pi^N) M_{\bar{V}^N} L^{-1}\pi^N\|_2^2\right)^{\frac{1}{2}} & \le Z_1^N,\\
    \left((Z_1^N)^2 + \|\bar{V}^N\|_1^2\max_{n \in \mathbb{N}_0^2\setminus I^N}\frac{1}{|l(\tilde{n})|^2}\right)^{\frac{1}{2}}  &\leq Z_1.
\end{aligned}
\end{equation}
Then we have $\left\|I_d - (B^N + \pi_N)\left(I_d + M_{\bar{V}^N}L^{-1}\right)\right\|_2 \leq Z_1$.
   \end{lemma}
The above quantities are related to finite-dimensional computations, and their computation (with interval arithmetic) is presented in \cite{julia_our_code}. It remains to compute the bound $\mathcal{Z}_u$ satisfying \eqref{def : Z1 periodic and Zu}.

Following the analysis of \cite{Cadiot2023}, we can write $\mathcal{Z}_u$ as 
\begin{align*}
    \mathcal{Z}_u = \max\left\{1, \|B^N\|_2\right\} \sqrt{\mathcal{Z}_{u,1}^2 + \mathcal{Z}_{u,2}^2}
\end{align*}
where $\mathcal{Z}_{u,1}$ and $\mathcal{Z}_{u,2}$ are defined as follows:
\begin{align*}
    \mathcal{Z}_{u,1} & = \sup_{\|u\|_2 =1} \left(\|\out\mathbb{L}_{0}^{-1}(\bar{v}^Nu)_{0}\|_2^2  +  2\sum_{n_2 \in \mathbb{N}} \|\out\mathbb{L}_{n_2}^{-1}(\bar{v}^Nu)_{n_2}\|_2^2\right)^\frac{1}{2},\\
    \mathcal{Z}_{u,2} & = \sup_{\|u\|_2 =1}\left(\|\cha (\mathbb{L}_{0}^{-1}- \Gamma^\dagger(L_{0}^{-1}))(\bar{v}^Nu)_{0}\|_2^2  +  2\sum_{n_2 \in \mathbb{N}} \|\cha (\mathbb{L}_{n_2}^{-1}- \Gamma^\dagger(L_{n_2}^{-1}))(\bar{v}^Nu)_{n_2}\|_2^2\right)^\frac{1}{2}.
\end{align*}
In particular, we have that $\mathbb{L}_{n_2}^{-1}$ is a convolution operator associated to the function $f_{n_2} \bydef \mathcal{F}^{-1}\left(\frac{1}{l_{n_2}}\right)$. Indeed, we have that 
\begin{align*}
    \mathbb{L}_{n_2}u = \mathcal{F}^{-1}\left(\mathcal{F}(\mathbb{L}_{n_2}u)\right) = \mathcal{F}^{-1}\left( \frac{1}{l_{n_2}} \hat{u}\right) = f_{n_2} * u.
\end{align*}
The results derived in \cite{Cadiot2023} require the computation of constants $C_{n_2}>0$ and $a_{n_2} >0$ such that 
\begin{equation}\label{eq : exp decay of fn2}
    |f_{n_2}(x)| \leq C_{n_2} e^{-a_{n_2}|x|}, ~~ \text{ for all } x \in \R.
\end{equation}
 If we can obtain explicit constants satisfying the above, then \cite{Cadiot2023} provides explicit formulas for the computation of $\mathcal{Z}_{u,1}$ and $\mathcal{Z}_{u,2}$. We derive such a result in the following lemma.
\begin{lemma}\label{lem : computation of f and exp decay}
    Let $n_2 \in \mathbb{N}_0$, and let us define $a_{n_2}$, $b_{n_2}$ and ${C}_{n_2}$ as 
    \begin{align}
        a_{n_2} &\bydef  \frac{\sqrt{4(1+\tn_2^2c^2)-c^4}}{2\sqrt{c^2-2\tilde{n}_2^2 + \sqrt{(c^2-2\tilde{n}_2^2)^2 + 4(1+c^2\tilde{n}_2^2)-c^4}}} ,  \\
    b_{n_2} &\bydef   \frac{1}{2} \sqrt{c^2-2\tilde{n}_2^2 + \sqrt{(c^2-2\tilde{n}_2^2)^2 + 4(1+c^2\tilde{n}_2^2)-c^4}},\\
    C_{n_2} &\bydef  \frac{1}{2(a_{n_2}^2 + b_{n_2}^2)^{\frac{1}{2}}\pi\sqrt{4(1+\tn_2^2c^2)-c^4}}.
    \end{align}
    Then, we obtain that ${C}_{n_2}$ and $a_{n_2}$ satisfy \eqref{eq : exp decay of fn2}. 
\end{lemma}

\begin{proof}

Let $n_2 \in \mathbb{N}_0$ and let us compute the inverse Fourier transform of $\frac{1}{l_{n_2}}$. We have
\begin{align*}
    f_{n_2}(x) = \int_{\R} \frac{1}{l_{n_2}(\xi)}e^{2\pi i \xi x}d\xi = \frac{1}{2\pi} \int_{\R} \frac{1}{l_{n_2}(\frac{\xi}{2\pi})} e^{i\xi x}d\xi.
\end{align*}
Now, notice that $\xi \mapsto l_{n_2}(\frac{\xi}{2\pi})$ has four roots given by $\pm \xi_1$ and $\pm \overline{\xi_1}$ where
\begin{align*}
    \xi_1 \bydef b_{n_2} - i a_{n_2}.
\end{align*}
In particular, we have that 
\begin{align*}
    \frac{1}{l_{n_2}(\frac{\xi}{2\pi})} = \frac{1}{i\sqrt{4(1+\tn_2^2c^2)-c^4}} \left( \frac{1}{\xi^2 - \xi_1^2} - \frac{1}{\xi^2 - \overline{\xi_1}^2} \right).
\end{align*}
Using Section 2.3.14 in \cite{transforms_poularikas_2010}, we obtain that 
\begin{align*}
     f_{n_2}(x) &=  \frac{1}{2\pi} \frac{1}{i\sqrt{4(1+\tn_2^2c^2)-c^4}} \int_{\R} \left(\frac{1}{\xi^2 - \xi_1^2} - \frac{1}{\xi^2 - \overline{\xi_1}^2}\right) e^{i\xi x}d\xi \\
     &=  \frac{1}{4\pi} \frac{1}{i\sqrt{4(1+\tn_2^2c^2)-c^4}} \left(\frac{e^{-i\xi_1|x|}}{i\xi_1} - \frac{e^{-\overline{i\xi_1}|x|}}{\overline{i\xi_1}} \right)
\end{align*}
for all $x \in \R$.
Then, we have that 
\begin{align*}
     f_{n_2}(x) &=  \frac{e^{-a_{n_2}|x|}}{4(a_{n_2}^2 + b_{n_2}^2)i\pi\sqrt{4(1+\tn_2^2c^2)-c^4}}\left(e^{-ib_{n_2}|x|}(a_{n_2} - ib_{n_2}) - e^{ib_{n_2}|x|}(a_{n_2} + i b_{n_2}) \right)\\
     &=  -\frac{e^{-a_{n_2}|x|}}{2(a_{n_2}^2 + b_{n_2}^2)\pi\sqrt{4(1+\tn_2^2c^2)-c^4}} \left( a_{n_2}\text{sign}(x)\sin(b_{n_2} x) + b_{n_2}\cos(b_{n_2}x) \right).
\end{align*}
Notice that $x \mapsto f_{n_2}(x)$ is an even function and that $f_{n_2}$ is smooth on $[0,\infty)$. In particular, we have that $C_{n_2}$ and $a_{n_2}$ satisfy \eqref{eq : exp decay of fn2}.
\end{proof}
Using the above, we obtain an explicit formula for $\mathcal{Z}_u$. 
\begin{lemma}\label{lem : Zu bound}
Let $E \in \ell^2$ the sequence of Fourier coefficients given as 
\begin{align}
    E_n \bydef  \frac{(C_{n_2})^2 a_{n_2}(-1)^{n_1}(1-e^{-4a_{n_2}d_1})}{d_1(4a_{n_2}^2 + (2\pi\tilde{n}_1)^2)} 
\end{align}
for all $n \in \mathbb{N}_0^2$. Moreover, let $a$ and $C(d_1)$ be defined as 
\begin{align*}
a \bydef \inf_{n_2 \in \mathbb{N}_0} a_{n_2} ~ \text{ and } ~ 
    C(d_1) \bydef 4d_1 + \frac{4e^{-ad_1}}{a(1-e^{-\frac{3}{2}ad_1})} + \frac{2}{a(1-e^{-2ad_1})}.
\end{align*}
Then, we obtain that
\begin{align*}
    \mathcal{Z}_{u,1} & \leq \bigg(2d_1 (\bar{V}^N,\bar{V}^N*E)_2\bigg)^{\frac{1}{2}},\\
    \mathcal{Z}_{u,2} & \leq  \bigg(2d_1 (\bar{V}^N,V^N*E)_2 + C(d_1) e^{-2a d_1}  (\bar{V}^N,\bar{V}^N*E)_2\bigg)^{\frac{1}{2}},
\end{align*}
and that $\mathcal{Z}_u = \max\left\{1, \|B^N\|_2\right\} \sqrt{\mathcal{Z}_{u,1}^2 + \mathcal{Z}_{u,2}^2}$ satisfies \eqref{def : Z1 periodic and Zu}.
\end{lemma}

\begin{proof}
Let $u \in L^2_e$ such that $\|u\|_{L^2_e} =1$ and let $v \bydef \bar{v}^N u$.
    Using the proof of Theorem 3.9 in \cite{Cadiot2023}, we have 
    \begin{align}\label{eq : Zu1 step 0}
        \|\out\mathbb{L}_{n_2}^{-1}v_{n_2}\|_2^2\leq C_{n_2}^2 \int_{\R \setminus (-d_1,d_1)} \left(\int_{-d_1}^{d_1} e^{-a_{n_2}|x-y|} |v_{n_2}(x)| dx \right)^2dy
    \end{align}
    for all $n_2 \in \mathbb{N}_0$, where we have used \eqref{eq : exp decay of fn2}. Noting that $v = \bar{v}^N u$, we have
    \begin{align*}
        v_{n_2} = \sum_{k \in \mathbb{Z}} \bar{v}^N_{|n_2-k|} u_{|k|}.
    \end{align*}
    for all $n_2 \in \mathbb{N}_0$. Going back to \eqref{eq : Zu1 step 0}, we get
    \begin{align*}
         \|\out\mathbb{L}_{n_2}^{-1}v_{n_2}\|_2^2 \leq  C_{n_2}^2 \int_{\R \setminus (-d_1,d_1)} \left(\sum_{k \in \mathbb{Z}}  \int_{-d_1}^{d_1} e^{-a_{n_2}|x-y|} |\bar{v}^N_{|n_2-k|}(x) u_{| k|}(x)| dx \right)^2dy.
    \end{align*}
    Now, applying the Cauchy-Schwarz inequality, we get
     \begin{align}\label{eq : Zu1 step 1}
     \nonumber
         &\|\out\mathbb{L}_{n_2}^{-1}v_{n_2}\|_2^2 \\ \nonumber
         \leq  ~ &  C_{n_2}^2 \int_{\R \setminus (-d_1,d_1)} \left(\sum_{k \in \mathbb{Z}}  \int_{-d_1}^{d_1} e^{-2a_{n_2}|x-y|} |\bar{v}^N_{|n_2 - k|}(x)|^2 dx \right) \left(\sum_{k \in \mathbb{Z}}  \int_{-d_1}^{d_1}  |u_{|k|}(x)|^2 dx \right)dy\\
         \leq ~& C_{n_2}^2 \int_{\R \setminus (-d_1,d_1)} \sum_{k \in \mathbb{Z}}  \int_{-d_1}^{d_1} e^{-2a_{n_2}|x-y|} |\bar{v}^N_{|n_2-k|}(x)|^2 dx dy
    \end{align}
    since $\|u\|_{2} =1.$ Employing Fubini's theorem, we have
    \begin{align*}
        \int_{\R \setminus (-d_1,d_1)} \sum_{k \in \mathbb{Z}}  \int_{-d_1}^{d_1} e^{-2a_{n_2}|x-y|} |\bar{v}^N_{|n_2-k|}(x)|^2 dx dy = 
 \int_{-d_1}^{d_1} \sum_{k \in \mathbb{Z}} |\bar{v}^N_{|n_2 -k|}(x)|^2  \int_{\R \setminus (-d_1,d_1)} \ e^{-2a_{n_2}|x-y|}  dy dx.
    \end{align*}
    Recalling the proof of Theorem 3.9 from \cite{Cadiot2023}, we get
\begin{align*}
    \int_{\R \setminus (-d_1,d_1)} \ e^{-2a_{n_2}|x-y|}  dy =  \frac{e^{-2a_{n_2}d_1}\cosh(2a_{n_2}x)}{a_{n_2}}.
\end{align*}
In particular, using Lemma 4.6 in \cite{cadiot2024whitham}, we obtain that $x \to \frac{e^{-2a_{n_2}d_1}\cosh(2a_{n_2}x)}{a_{n_2}}$ has a Fourier coefficients representation on $(-d_1,d_1)$ given by the sequence $\tilde{E}$, where
\begin{align*}
    \tilde{E}_{n_1} \bydef \frac{a_{n_2}(-1)^{n_1}(1-e^{-4a_{n_2}d_1})}{d_1(4a_{n_2}^2 + (2\pi\tilde{n}_1)^2)}
\end{align*}
    for all $n_1 \in \mathbb{N}_0$ and all $n_2 \in \mathbb{N}_0$. Applying Parseval's identity, this implies that
    \begin{align}\label{eq : Zu1 step 2}
    \nonumber
         \mathcal{Z}_{u,1}^2 &\leq \sum_{k \in \mathbb{Z}}\sum_{n_2 \in \mathbb{Z}} \int_{-d_1}^{d_1}  |\bar{v}^N_{|n_2 - k|}(x)|^2  \frac{ C_{|n_2|}^2 e^{-2a_{n_2}d_1}\cosh(2a_{n_2}x)}{a_{n_2}} dx\\
         &= 2d_1  (\bar{V}^N,\bar{V}^N*E)_2.
    \end{align}
   This concludes the inequality for $\mathcal{Z}_{u,1}$.

\vspace{0.2cm}

For the computation of $\mathcal{Z}_{u,2}$ we use the formulas provided in \cite{Cadiot2023}. Indeed, applying the proof of Theorem 3.9 again, we can define $g_n$ as 
    \begin{align*}
        g_n \bydef \frac{1}{2d_1} \int_{-d_1}^{d_1} \mathbb{L}_{n_2} v_{n_2}(x) e^{-2\pi i \tilde{n}_1 x}dx 
    \end{align*}
for all $n \in \mathbb{Z}^2$, where $v_{n_2} = v_{-n_2}$ if $n_2 < 0$. Utilizing the proof of Theorem 3.9 in \cite{Cadiot2023}, we have that 
\begin{align*}
    \mathcal{Z}_{u,2}^2 = 2d_1 \sum_{n \in \mathbb{Z}^2} |g_n|^2 = \int_{\R \setminus ((-d_1,d_1) \cup (-d_1,d_1)+2d_1n_1)} \mathbb{L}_{n_2}v_{n_2}(x) \mathbb{L}_{n_2}v_{n_2}(x-2d_1n_1)dx.
\end{align*}
In particular, note that for $n_1=0$, we have 
\begin{align*}
  \sum_{n_2 \in \mathbb{Z}}  \int_{\R \setminus ((-d_1,d_1) \cup (-d_1,d_1)+2d_1n_1)} \mathbb{L}_{n_2}v_{n_2}(x) \mathbb{L}_{n_2}v_{n_2}(x-2d_1n_1)dx =  \sum_{n_2 \in \mathbb{Z}}  \int_{\R \setminus ((-d_1,d_1)} \left|\mathbb{L}_{n_2}v_{n_2}(x)\right|^2 dx = \mathcal{Z}_{u,1}^2.
\end{align*}
With the help of the proof of Lemma 6.5 in  \cite{Cadiot2023}, we obtain that 
\begin{align*}
    \mathcal{Z}_{u,2}^2 \leq \mathcal{Z}_{u,1}^2 +  \sum_{n \in \mathbb{Z}^2, n_1 \neq 0} C_{n_2}^2 \int_{(-d_1,d_1)^2} |v_{n_2}(x) v_{n_2}(z)| I_n(x,z) dxdz 
\end{align*}
where 
\[
I_{n}(x,z) \bydef \int_{\R \setminus ((-d_1,d_1) \cup (-d_1,d_1)+2d_1n_1)} e^{-a_{n_2}|y-x|} e^{-a_{n_2}|2d_1 n_1 + z-y|} dy.
\]
We conclude the proof following the steps of the proof of Lemma 6.5 in \cite{Cadiot2023} and using the above computations concerning the bound $\mathcal{Z}_{u,1}$.
\end{proof}

\section{Stability Analysis}\label{sec:stability}
In this section, we aim at investigating the stability of traveling wave solutions to \eqref{eq : suspension bridge equation} with boundary conditions \eqref{eq : neumann boundary conditions}, which is the traveling wave formulation associated with \eqref{eq:SBnowave}. The existence of these traveling waves will be constructively established in Section \ref{sec : results}. In particular, we assume that we know the existence of  a solution $\tilde{u} \in \mathcal{H}$ to \eqref{eq : suspension bridge equation} such that 
\begin{equation}
    \|\tilde{u} - \bu\|_{\mathcal{H}} \leq r_0,
\end{equation}
where $\ou$ is an approximate solution constructed as in \eqref{def : approx solution} and $r_0$ is explicit. In order to study stability, we heavily rely on the framework developed in \cite{SB_nagatou}. In fact, the authors of \cite{SB_nagatou} investigated the orbital stability of traveling wave solutions in the 1D version of \eqref{eq : suspension bridge equation}. One can easily show that the analysis of Sections 3 and 4 of \cite{SB_nagatou} generalizes to  the 2D version \eqref{eq : suspension bridge equation}. We recall some notations for completeness.

Let $\Omega \bydef \R \times (-d_2,d_2)$ be our spatial domain. Then,
let $X \bydef H^2(\Omega) \times L^2(\Omega)$ and let~$X^* \bydef H^{-2}(\Omega) \times L^2(\Omega)$ be its dual. Moreover, let us define the functional $E : X \to \R$ as 
\begin{equation}
    E\begin{pmatrix}
        u\\
        v
    \end{pmatrix} \bydef \int_{\Omega}\left[ \frac{1}{2}v(x)^2 + \frac{1}{2}(\Delta u(x))^2 + e^{u(x)} - u(x) -1 \right]dx 
\end{equation}
Then, \eqref{eq:SBnowave} can be written in the form
\begin{equation}\label{eq : time evolution weak form}
    \frac{d}{dt} U = J E'(U)
\end{equation}
where $U = \begin{pmatrix}
        u\\
        v
    \end{pmatrix}$, $J:L^2(\Omega) \times H^2(\Omega) \to X$ defined by $J\begin{pmatrix}
        u\\
        v
    \end{pmatrix} = \begin{pmatrix}
        v\\
        -u
    \end{pmatrix}$ and $E'\begin{pmatrix}
        u\\
        v
    \end{pmatrix} = \begin{pmatrix}
        \Delta^2 u + e^{u}-1\\
        v
    \end{pmatrix} \in X^*$. Let $T(s) : X \to X$ be the translation by $s \in \mathbb{R}$ in the direction $x_1$, that is $T(s)U(x) = U(x_1+s,x_2)$. Defining the so-called bound state
    \begin{align}\label{eq : definition of big U tilde}
        \tilde{U} \bydef \begin{pmatrix}
            \tilde{u}\\
            c \partial_{x_1} \tilde{u}
        \end{pmatrix},
    \end{align}
    we find that $U(t):=T(ct)\tilde{U}$ provides a solution to \eqref{eq : time evolution weak form}.

    We are now in a position to introduce the notation of orbital stability used in \cite{SB_nagatou}.
    \begin{definition}
        Let $\tilde{u}$ be a traveling wave solution to \eqref{eq : suspension bridge equation} in $\mathcal{H}$ and let $\tilde{U}$ be defined as in \eqref{eq : definition of big U tilde}. Then $\tilde{u}$ is said to be \textit{orbitally stable} if the solution $U$ to \eqref{eq : time evolution weak form} with initial condition $U(0) = U_0$ exists for all positive times and if for all $\epsilon>0$, there exists $\delta>0$ such that if $\|U_0 - \tilde{U}\|_{X}<\delta$, then $\inf_{s \in \R} \|U(t) - T(s) \tilde{U} \|_{X} < \epsilon.$ Furthermore, $\tilde{u}$ is called \textit{orbitally unstable} if it is not orbitally stable.
    \end{definition}

Let $n(D\mathbb{F}(\tilde{u}))$ be the number of negative eigenvalues of $D\mathbb{F}(\tilde{u})$, counted with multiplicity. Then, the authors of \cite{SB_nagatou} derived in Theorem 1 sufficient conditions under which orbital (in)stability is achieved. This result naturally generalizes to the 2D case, since the steps of the proof are essentially the same, as also observed by the authors in \cite{SB_nagatou}. We state the two-dimensional version in the following lemma.

\begin{lemma}\label{lem : conditions for stability}
    Let $\tilde{w} \in \mathcal{H}$ be defined as 
    \begin{align}\label{def : definition of tilde v}
        \tilde{w} = -2 c D\mathbb{F}(\tilde{u})^{-1} \partial_{x_1}^2 \tilde{u}.
    \end{align}
    Moreover, define the constant $\theta$ as follows
    \begin{align}\label{def : definition of theta}
        \theta \bydef \int_{\Omega} (\tilde{u} + 2c \tilde{w})\partial_{x_1}^2\tilde{u} . 
    \end{align}
    If $n(D\mathbb{F}(\tilde{u})) =0$, or if $\theta >0$ and  $n(D\mathbb{F}(\tilde{u})) = 1$, then $\tilde{u}$ is orbitally stable. If $\theta>0$ and $n(D\mathbb{F}(\tilde{u}))$ is even, or if $\theta <0$ and $n(D\mathbb{F}(\tilde{u}))$ is odd, then $\tilde{u}$ is orbitally unstable.
\end{lemma}

\begin{remark}
    A solution is not only valid for a fixed parameter $c$, but also persists under variations of $c$. Differentiating with respect to $c$ yields $\tilde{w}$ in Lemma \ref{lem : conditions for stability}, which is called $v_c$ in \cite{SB_nagatou}. To determine (in)stability, the authors of \cite{SB_nagatou} introduce a functional $d(c)$ and compute its second derivative $d''(c)$ in Section 4.3 of their paper, which  corresponds to $\theta$ in our Lemma \ref{lem : conditions for stability} after applying integrations by parts.
\end{remark}

The previous lemma provides sufficient conditions for proving (in)stability, which depend on the spectrum of the Jacobian $D\mathbb{F}(\tilde{u})$ and on the sign of $\theta$. Consequently, our objective is now to enclose the spectrum of $D\mathbb{F}(\tilde{u})$ and to provide a computer-assisted approach for enclosing the value of $\theta.$

\subsection{Enclosure of the spectrum of $D\mathbb{F}(\tilde{u})$}\label{sec : enclosure of spectrum}

In this section, we present a computer-assisted strategy for the enclosure of the spectrum of $D\mathbb{F}(\tilde{u}):  \mathcal{H} \to L^2(\Omega)$, denoted as $\sigma(D\mathbb{F}(\tilde{u}))$.  First, notice that $D\mathbb{F}(\tilde{u})$ viewed as an operator on $L^2(\Omega)$ with domain $\mathcal{H}$ is self-adjoint. This implies that its spectrum is real-valued. Consequently, we restrict our study of the spectrum to the real line.

In order to obtain explicit enclosures on $\sigma(D\mathbb{F}(\tilde{u}))$,  we follow the set-up introduced in \cite{cadiot2025stabilityanalysislocalizedsolutions}. In particular, we want to construct a homotopy allowing the study of the spectrum of $D\mathbb{F}(\tilde{u})$ based on the one of $DF(\bbu)$.

To begin with, we control the essential part $ \sigma_{ess}(D\mathbb{F}(\tilde{u})$ of the spectrum.
In particular, using \cite{cadiot2025stabilityanalysislocalizedsolutions}, we have that the essential spectrum can be computed using the Fourier transform of $\mathbb{L}$ as follows:
\begin{align*}
    \sigma_{ess}(D\mathbb{F}(\tilde{u}) = \{l_{n_2}(\xi), ~ \xi \in \R, n_2 \in \mathbb{N}_0\} \subset \left[1 - \frac{c^4}{4}, \infty\right).
\end{align*}

Then, using Lemma 2.2 of \cite{cadiot2025stabilityanalysislocalizedsolutions}, we have the following decomposition of the spectrum
\begin{align}
    \sigma(D\mathbb{F}(\tu)) = \sigma_{ess}(D\mathbb{F}(\tu)) \cup Eig(D\mathbb{F}(\tu)).
\end{align}
Consequently, it remains to control the eigenvalues of $D\mathbb{F}(\tu)$ away from the essential spectrum.  First, we construct a lower bound for the eigenvalues.
\begin{lemma}\label{lem : lambda min}
    Let $\lambda_{min}$ be defined as 
    \begin{equation}
        \lambda_{min} \bydef - e^{\kappa_2 r_0}\|e^{\bbu}\|_1  - \frac{c^4}{4}.
    \end{equation}
    Then $(-\infty, 1-\frac{c^4}{4})\cap Eig(D\mathbb{F}(\tu)) \subset [\lambda_{min}, 1-\frac{c^4}{4})$.
\end{lemma}

\begin{proof}
    Let $\lambda \in  (-\infty, 1-\frac{c^4}{4})\cap Eig(D\mathbb{F}(\tu))$. Then, there exists $u \in \mathcal{H}$, $u \neq 0$, such that 
    \[
    \mathbb{L} u + (e^{\tilde{u}}-1)u = \lambda u.
    \]
    Using that $\lambda < 1-\frac{c^4}{4}$, we have that $|l_{n_2}(\xi) - \lambda| \geq 1- \frac{c^4}{4} - \lambda$ for all $\xi \in \R$ and all $n_2 \in \mathbb{N}_0$.
This implies that 
\begin{align*}
    \left(1- \frac{c^4}{4} - \lambda \right)\|u\|_2 \leq \|\mathbb{L}u - \lambda u\|_2 \leq  \|e^{\tilde{u}}u -u\|_2 \leq \|e^{\tilde{u}}u\|_2 + \|u\|_2.
\end{align*}
Now, we have
\begin{align*}
    \|e^{\tilde{u}}u\|_2 = \|e^{\tilde{u}-\bu} e^{\bu}u\|_2 &\leq \|e^{\tilde{u}-\bu}\|_\infty \|e^{\bu}\|_\infty \|u\|_2\\
   & \leq e^{\kappa_2 \|\tilde{u}-\bu\|_{\mathcal{H}}} \|e^{\bbu}\|_1 \|u\|_2\\
  &  \leq e^{\kappa_2 r_0} \|e^{\bbu}\|_1 \|u\|_2.
\end{align*}
This concludes the proof after a short algebraic manipulation.
\end{proof}

Having established the lower bound $\lambda_{\mathrm{min}}$, we now restrict our work to the interval $\mathcal{J} \bydef [\lambda_{min}, \delta_0]$ where $0<\delta_0< 1-\frac{c^4}{4}$. Then, the above results provide that the negative eigenvalues of $D\mathbb{F}(\tilde{u})$ must be contained in $\mathcal{J}$. Using the analysis derived in \cite{cadiot2025stabilityanalysislocalizedsolutions}, we provide rigorous enclosures for the set $\mathcal{J}\cap Eig(D\mathbb{F}(\tilde{u}))$, based on the spectrum of $DF(\bbu).$ Our approach is computer-assisted and follows the methodology derived in Sections 3 and 4 of \cite{cadiot2025stabilityanalysislocalizedsolutions}. Before presenting such an approach, we present a simplification by separating even and odd subspaces.

Using the ideas of Section 5.1 in \cite{cadiot2025stabilityanalysislocalizedsolutions},  we notice that $L^2$ can be decomposed as follows
\begin{align*}
L^2 = L^2_{e} \oplus L^2_{o},
\end{align*}
where 
\begin{align}
    L^2_o  &\bydef \left\{u \in L^2, ~ u(x_1,x_2) = - u(-x_1,x_2)  \text{ for all } x \in \R^2\right\}.\label{def : Hc and Hs}
\end{align}
 Moreover, given $\tilde{u} \in \mathcal{H}$, we have that 
\begin{align*}
    D\mathbb{F}(\tilde{u})v_{i} \in L^2_{i}
\end{align*}
for all $v_{i} \in L^2_{i}$ and all $i \in \{e,o\}$. This implies that if $u \in L^2$ satisfies 
\begin{align*}
    D\mathbb{F}(\tilde{u} )u = \lambda u
\end{align*}
and we decompose $u = u_{e} + u_{o}$, where $u_{i} \in L^2_{i}$, then 
\[
 D\mathbb{F}(\tilde{u} )u_{i} = \lambda u_{i}
\]
for all $i \in \{e,o\}$. Consequently, we can investigate the spectrum of $D\mathbb{F}_{i}(\tilde{u} ) : L^2_{i} \to L^2_{i}$, the restriction of $D\mathbb{F}(\tilde{u})$ to $L^2_{i} \to L^2_{i}$, for all $i \in \{e,o\}$ and obtain the following decomposition 
\begin{equation}\label{eq : identity symmetry spectrum SH}
    \sigma\left(D\mathbb{F}(\tilde{u})\right) =  \sigma\left(D\mathbb{F}_{e}(\tilde{u})\right) \cup  \sigma\left(D\mathbb{F}_{o}(\tilde{u})\right).
\end{equation}
Then, in practice, we control the spectrum of each $D\mathbb{F}_{i}(\tilde{u})$ and use the identity \eqref{eq : identity symmetry spectrum SH} to conclude. This disjunction allows to optimize numerical memory usage by considering smaller operators.

Going back to our strategy for enclosing the eigenvalues of $D\mathbb{F}(\tilde{u})$,  we introduce a pseudo-diagonalization for $DF(\bbu)$ given as 
    \begin{align*}
        \mathcal{D} \bydef P^{-1} DF(\bbu) P,
    \end{align*}
    for some invertible infinite matrix $P = P^N + \pi_{N}$, with $P^{N} = \pi^{N} P^N \pi^{N}$, whose inverse is given as $P^{-1} =  \pi^{N}(P^N)^{-1} \pi^{N} + \pi_{N}$. $(P^{N})^{-1}$ has to be understood as the inverse of $P^{N} : \pi^{N} \ell^2 \to \pi^{N}\ell^2$. In practice, columns of  $P$ are approximate eigenvectors of $DF(\bbu)$, so that $\mathcal{D}$ is close to being diagonal. We define $S$ as the diagonal part of the matrix $\mathcal{D}$ and $R = \mathcal{D} - S$. Moreover, we define $\lambda_k$ as the diagonal entries of $S$, that is, 
    \begin{equation}\label{def : lambda n}
          (SU)_k = \lambda_k U_k \text{ for all } k \in \mathbb{N}_0^2.
    \end{equation}
    Our goal will be to prove that the eigenvalues of $D\mathbb{F}(\tilde{u})$ are contained in a neighborhood of the sequence $(\lambda_k)$.
    
  Before delving further into the enclosure of the spectrum of $D\mathbb{F}(\tilde{u})$, we introduce the following auxiliary lemma, which will be useful for our estimations.
\begin{lemma}\label{lem : estimation fn2 mu}
    Let $n_2 \in \mathbb{N}_0$, and let us define $a_{n_2}(\mathcal{J})$, $b_{n_2}(\mathcal{J})$ and ${C}_{n_2}(\mathcal{J})$  as (cf. Lemma \ref{lem : computation of f and exp decay})
    \begin{align}
        a_{n_2}(\mathcal{J}) &\bydef  \frac{\sqrt{4(1+\tn_2^2c^2-\delta_0)-c^4}}{2\sqrt{c^2-2\tilde{n}_2^2 + \sqrt{(c^2-2\tilde{n}_2^2)^2 + 4(1+c^2\tilde{n}_2^2-\delta_0)-c^4}}}  \label{eq:an_stab}, \\
    b_{n_2}(\mathcal{J}) &\bydef   \frac{1}{2} \sqrt{c^2-2\tilde{n}_2^2 + \sqrt{(c^2-2\tilde{n}_2^2)^2 + 4(1+c^2\tilde{n}_2^2-\delta_0)-c^4}},\\
    C_{n_2}(\mathcal{J}) &\bydef  \frac{1}{2(a_{n_2}(\mathcal{J})^2 + b_{n_2}(\mathcal{J})^2)^{\frac{1}{2}}\pi\sqrt{4(1+\tn_2^2c^2-\delta_0)-c^4}}.
    \end{align}
     Given $\mu \in \mathcal{J}$, we define $f_{n_2,\mu} \bydef \mathcal{F}^{-1}(\frac{1}{l_{n_2} - \mu})$.
    Then, 
    we obtain that
    \begin{equation}\label{eq : exp decay of fn2 stability}
   \sup_{\mu \in \mathcal{J}} |f_{n_2,\mu}(x)| \leq C_{n_2}(\mathcal{J}) e^{-a_{n_2}(\mathcal{J})|x|}, ~~ \text{ for all } x \in \R.
\end{equation}
\end{lemma}

\begin{proof}
    The proof is very similar to the one of Lemma \ref{lem : computation of f and exp decay}. Indeed, we compute $f_{n_2,\mu}$ as in the proof of Lemma \ref{lem : computation of f and exp decay}, and obtain that 
    \begin{align*}
        |f_{n_2,\mu}(x)| \leq C_{n_2}(\mu) e^{-a_{n_2}(\mu)|x|}, ~~ \text{ for all } x \in \R
    \end{align*}
where
    \begin{align*}
        a_{n_2}(\mu) &\bydef  \frac{\sqrt{4(1+\tn_2^2c^2-\mu)-c^4}}{2\sqrt{c^2-2\tilde{n}_2^2 + \sqrt{(c^2-2\tilde{n}_2^2)^2 + 4(1+c^2\tilde{n}_2^2-\mu)-c^4}}} ,  \\
    b_{n_2}(\mu) &\bydef   \frac{1}{2} \sqrt{c^2-2\tilde{n}_2^2 + \sqrt{(c^2-2\tilde{n}_2^2)^2 + 4(1+c^2\tilde{n}_2^2-\mu)-c^4}},\\
    C_{n_2}(\mu) &\bydef  \frac{1}{2(a_{n_2}(\mu)^2 + b_{n_2}(\mu)^2)^{\frac{1}{2}}\pi\sqrt{4(1+\tn_2^2c^2-\mu)-c^4}}.
    \end{align*}
    Subsequently, one can prove that both $a_{n_2}$ and $b_{n_2}$ are decreasing functions of $\mu$, leading to the desired result.
\end{proof}

Next, we aim at applying Theorem 4.3 from \cite{cadiot2025stabilityanalysislocalizedsolutions} in order to enclose the eigenvalues of  $D\mathbb{F}(\tilde{u})$ in intervals centered at the $\lambda_k$'s. Note that the set-up is slightly different since \cite{cadiot2025stabilityanalysislocalizedsolutions} focuses on  the domain $\R^m$, where \eqref{eq : suspension bridge equation} is posed on $\Omega = \R \times (-d_2,d_2)$. However, a very similar analysis applies and we summarize the main result of \cite{cadiot2025stabilityanalysislocalizedsolutions} in the case of $\Omega$.
    \begin{lemma}\label{lem : enclosure spectrum}
       Let $t > -\lambda_{min}$ and let $\mathcal{J} = [\lambda_{min},\delta_0]$. We define $E = (E_n)_{n \in \mathbb{N}_0^2}$ as 
       \begin{equation}
         E_n \bydef    \frac{(C_{n_2}(\mathcal{J}))^2 a_{n_2}(\mathcal{J})(-1)^{n_1}(1-e^{-4a_{n_2}(\mathcal{J})d_1})}{d_1(4a_{n_2}(\mathcal{J})^2 + (2\pi\tilde{n}_1)^2)}. 
       \end{equation}
       Furthermore,  define $a(\mathcal{J})$ and $C(d_1)$ as 
       \begin{align*}
           a(\mathcal{J}) \bydef \inf_{n_2 \in \mathbb{N}_0} a_{n_2}(\mathcal{J}) ~ \text{ and } ~ 
    C(d_1) \bydef 4d_1 + \frac{4e^{-a(\mathcal{J})d_1}}{a(\mathcal{J})(1-e^{-\frac{3}{2}a(\mathcal{J})d_1})} + \frac{2}{a(\mathcal{J})(1-e^{-2a(\mathcal{J})d_1})}.
       \end{align*}
       We introduce various constants satisfying the following  inequalities:
       \begin{align*}
          \mathcal{Z}_{u,1} &\geq 4d_1 (\bar{V}, \bar{V}*E),\\
         \mathcal{Z}_{u,2} &\geq  4d_1 (\bar{V},\bar{V}*,E)_2 + 2C(d_1) e^{-2a(\mathcal{J}) d_1}  (\bar{V},\bar{V}*E)_2\\
         \mathcal{Z}_{u,3} &\geq   \|\pi^N (S+ tI)^{-1}P^{-1}(L-\delta_0 I)\|_2 \mathcal{Z}_{u,2},\\
         \mathcal{C}_1 &\geq \frac{\|e^{\bbu}\|_1}{1-\delta_0 - \frac{c^4}{4}} \frac{e^{\kappa_2 r_0}-1}{r_0},\\
         \mathcal{C}_2 &\geq \|e^{\bbu}\|_1 \frac{\|\pi^N (S+ tI)^{-1}P^{-1}(L-\delta_0 I)\|_2}{1-\delta_0 - \frac{c^4}{4}} \frac{e^{\kappa_2 r_0}-1}{r_0}, \\
           Z_{1,1} &\geq  \frac{1}{1-\delta_0 - \frac{c^4}{4}}\|\pi_N (L - \delta_0 I)^{-1}R \pi^N\|_2, ~~~~ Z_{1,2} \geq  \frac{1}{1-\delta_0 - \frac{c^4}{4}}\|\pi_N (L - \delta_0 I)^{-1}R\pi_N\|_2,\\
       Z_{1,3} &\geq\| \pi^N(S+ tI)^{-1}R\pi^N\|_2, ~~~~ Z_{1,4} \geq \| \pi^N(S+ tI)^{-1}R\pi_N\|_2.
       \end{align*}
        Assuming
  \begin{equation}\label{eq : condition C1 r0}
      \mathcal{C}_1r_0 <1,
  \end{equation}
 we define $\beta_1$ as 
$\beta_1 \bydef\frac{\mathcal{Z}_{u,1}+\mathcal{C}_1r_0}{1-\mathcal{C}_1r_0}.$
Furthermore, we assume that
\begin{equation}\label{eq : condition for kappa1 and C1 r0}
    1-Z_{1,2}-\mathcal{Z}_{u,2} - (1+\beta_1^2)^\frac{1}{2}\mathcal{C}_1r_0 > 0.
\end{equation}
 Then, we define
\begin{equation*}
\begin{aligned}
\beta_2 &\bydef \frac{Z_{1,1} + (\mathcal{Z}_{u,2} + (1+\beta_1^2)^\frac{1}{2}\mathcal{C}_1r_0)\|P^N\|_2}{1-Z_{1,2}-\mathcal{Z}_{u,2} - (1+\beta_1^2)^\frac{1}{2}\mathcal{C}_1r_0},\\
\epsilon_n^{(q)} &\bydef |\lambda_n + t|\left(Z_{1,3} + Z_{1,4} \frac{Z_{1,1} + \mathcal{Z}_{u,2}\|P^N\|_2}{1-Z_{1,2}-\mathcal{Z}_{u,2}} +  \mathcal{Z}_{u,3}\left(\|P^N\|_2 +   \frac{Z_{1,1} + \mathcal{Z}_{u,2}\|P^N\|_2}{1-Z_{1,2}-\mathcal{Z}_{u,2}} \right) \right), \\
      \epsilon_n^{(\infty)} &\bydef |\lambda_n + t|\left(Z_{1,3} + Z_{1,4} \beta_2 + \left(\mathcal{Z}_{u,3} + \mathcal{C}_2r_0(1+\beta_1^2)^\frac{1}{2} \right)\left(\|P^N\|_2 +   \beta_2 \right)   \right),\\
      \epsilon_n &\bydef \max\{\epsilon_n^{(q)},\epsilon_n^{(\infty)}\}
\end{aligned}
\end{equation*}
for all $n \in \mathbb{N}_0^2$. Let  $k \in \mathbb{N}$ and $I \subset \mathbb{N}_0^2$ such that $|I| = k$.

If $\cup_{n \in I} [\lambda_n - \epsilon_n, \lambda_n + \epsilon_n] \subset \mathcal{J}$ and $\left(\cup_{n \in I} [\lambda_n - \epsilon_n, \lambda_n + \epsilon_n]\right) \bigcap\left( \cup_{n \in \mathbb{N}^2_0 \setminus I} [\lambda_n - \epsilon_n, \lambda_n + \epsilon_n]\right) = \varnothing$, then there are exactly $k$ eigenvalues of $D\mathbb{F}(\tilde{u})$ in $\cup_{n \in I} [\lambda_n - \epsilon_n, \lambda_n + \epsilon_n] \subset  \mathcal{J}$ counted with multiplicity.
    \end{lemma}

\begin{proof}
    The proof of the lemma follows the ones of Lemmas 4.1 and 4.2 in \cite{cadiot2025stabilityanalysislocalizedsolutions}. Following Lemma 4.1 in \cite{cadiot2025stabilityanalysislocalizedsolutions}, we want to compute $\mathcal{Z}_{u,1},\mathcal{Z}_{u,2}$ such that 
    \begin{align*}
         \mathcal{Z}_{u,1} &\geq  \sup_{\mu \in \mathcal{J}} \left\|\mathbb{1}_{\Omega \setminus \Omega_0} \left(\mathbb{L} - \mu I\right)^{-1} D\mathbb{G}(\bu)\right\|_2,\\
\mathcal{Z}_{u,2} & \geq \sup_{\mu \in \mathcal{J}} \| \mathbb{1}_{\Omega_0} \left(\Gamma^\dagger\left((L- \mu I )^{-1}\right) - (\mathbb{L}- \mu I)^{-1}\right)D\mathbb{G}(\bu) \|_2.
    \end{align*}
    Noticing that $\left(\mathbb{L}_{n_2} - \mu I\right)^{-1}$ is the convolution operator associated to the function $f_{n_2,\mu}$ defined in Lemma \ref{lem : estimation fn2 mu}, we combine Lemma \ref{lem : Zu bound} and Lemma \ref{lem : estimation fn2 mu} to obtain that
    \small{
     \begin{align*}
         2d_1 (\bar{V}, \bar{V}*E) &\geq  \sup_{\mu \in \mathcal{J}} \left\|\mathbb{1}_{\Omega \setminus \Omega_0} \left(\mathbb{L} - \mu I\right)^{-1} D\mathbb{G}(\bu)\right\|_2,\\
2d_1 (\bar{V},\bar{V}*E)_2 + C(d_1) e^{-2a d_1}  (\bar{V},\bar{V}*E)_2 & \geq \sup_{\mu \in \mathcal{J}} \| \mathbb{1}_{\Omega_0} \left(\Gamma^\dagger\left((L- \mu I )^{-1}\right) - (\mathbb{L}- \mu I)^{-1}\right)D\mathbb{G}(\bu) \|_2.
    \end{align*}
    }\normalsize
In particular, this validates our choice for $\mathcal{Z}_{u,1},\mathcal{Z}_{u,2}$. In fact, as explained in the proof of Theorem 5.2 in \cite{cadiot2025stabilityanalysislocalizedsolutions}, the extra factor $2$ in the bounds $\mathcal{Z}_{u,1},\mathcal{Z}_{u,2}$ allows them to satisfy the inequalities for $\mathcal{Z}_{u,1}^{(q)},\mathcal{Z}_{u,2}^{(q)}$ in Lemma 4.2 in \cite{cadiot2025stabilityanalysislocalizedsolutions}. This means that we can choose $\mathcal{Z}_{u,1}^{(q)} =\mathcal{Z}_{u,1}$ and $\mathcal{Z}_{u,2}^{(q)} = \mathcal{Z}_{u,2}$ in the aforementioned lemma.
    For the bound $\mathcal{Z}_{u,3}$, we need 
    \begin{align*}
        \mathcal{Z}_{u,3} & \geq \sup_{\mu \in \mathcal{J}}  \|\pi^N (S+ tI)^{-1}P^{-1}(L-\mu I)\|_2 \| \mathbb{1}_{\Omega_0} \left(\Gamma^\dagger\left((L- \mu I )^{-1}\right) - (\mathbb{L}- \mu I)^{-1}\right)D\mathbb{G}(\bu) \|_2.
    \end{align*}
    Let $\mu \in \mathcal{J}$, then we have
    \begin{align*}
        \|\pi^N (S+ tI)^{-1}P^{-1}(L-\mu I)\|_2  \leq  \|\pi^N (S+ tI)^{-1}P^{-1}(L-\delta_0 I)\|_2 \|(L-\delta_0 I)^{-1}(L-\mu I)\|_2.
    \end{align*}
    Then, we have
    \begin{align*}
        \|(L-\delta_0 I)^{-1}(L-\mu I)\|_2 = \sup_{n \in \mathbb{N}_0^2} \frac{l(\tilde{n})-\mu}{l(\tilde{n})-\delta_0} = \sup_{n \in \mathbb{N}_0^2} 1 + \frac{\delta_0-\mu}{l(\tilde{n})-\delta_0} =  \frac{ 1 -\frac{c^4}{4}-\mu}{1 -\frac{c^4}{4}-\delta_0}
    \end{align*}
    using Lemma \ref{lem : results on l}. Then, recalling $C_{n_2}(\mu)$ from the proof of Lemma \ref{lem : estimation fn2 mu}, one can prove that $\mu \mapsto C_{n_2}(\mu)\frac{ 1 -\frac{c^4}{4}-\mu}{1 -\frac{c^4}{4}-\delta_0}$ is increasing in $\mu$ and consequently the maximum is reached at $\mu = \delta_0$ when considering $\mu \in \mathcal{J}$. This concludes the inequality for $\mathcal{Z}_{u,3}$. The computations of $\mathcal{C}_1$ and $\mathcal{C}_2$ follow directly from the proof of Lemma \ref{lem : Z2 bound}. In particular, we also use that $\|(\mathbb{L}-\mu I)^{-1}\|_2 \leq \frac{1}{1-\frac{c^4}{4}-\delta_0}$ for all $\mu \in \mathcal{J}$. The rest of the proof follows from \cite{cadiot2025stabilityanalysislocalizedsolutions}.
\end{proof}


The previous lemma provides explicit formulas for the computation of intervals, called Gershgorin intervals, $[\lambda_j - \epsilon_j, \lambda_j + \epsilon_j]$ containing the eigenvalues of $D\mathbb{F}(\tilde{u}).$ In fact, it also allows to count the number of eigenvalues in a union of such intervals (if disjoint from the rest of the intervals). This enables us to exactly count the number of negative eigenvalues of $D\mathbb{F}(\tilde{u})$, having in mind the application of Lemma \ref{lem : conditions for stability}.

\subsection{Enclosure of $\theta$}\label{sec : enclosure of theta}

In this section, we provide a computer-assisted strategy for enclosing the value of $\theta$ defined in Lemma \ref{lem : conditions for stability}. In fact, we derive the following result.

\begin{lemma}\label{lem : enclosure of theta}
Let $r_0 >0$  and let $\tilde{u} \in \mathcal{H}$ be a solution to \eqref{eq : suspension bridge equation} such that $\|\tilde{u}-\bar{u}\|_{\mathcal{H}} \leq r_0$, where $\bar{u} \in \mathcal{H}$ is constructed as in \eqref{def : approx solution}.
Choose $\overline{W} \in \pi^{N}\ell^2$ such that $\overline{w} \bydef \gamma^\dagger(\overline{W}) \in \mathcal{H}$.
    Let $\epsilon_0, \epsilon >0$ be defined as 
    \begin{align}
    \epsilon_0 &\bydef \kappa_1 r_0 + 4c^2 \|D\mathbb{F}(\tilde{u})^{-1}\|_2 \left(\|\partial_{x_1}\bu + \frac{1}{2c} D\mathbb{F}(\bu) \overline{w}\|_2 + \frac{r_0}{2-c^2}  + \frac{1}{2c}(1-e^{\kappa_2r_0}) \|e^{\bbu}\overline{W}\|_1\right), \nonumber\\
        \epsilon &\bydef \|\bu+ 2c \overline{w}\|_2  \|\partial_{x_1}^2 \mathbb{L}^{-1}\|_2 r_0 + \epsilon_0 \left(\|\partial_{x_1}^2\bu\|_2  + \frac{r_0}{2-c^2} \right).
    \end{align}
    Then, we obtain that 
    \begin{align}
        \theta \in [\theta_0-\epsilon, \theta_0 + \epsilon],
    \end{align}
    where 
    \[
    \theta_0 \bydef |\om| \left((\bbu+ 2c \overline{W}), \partial_{x_1}^2\bbu\right)_2, ~~ \text{ and where } (\partial_{x_1}^2\bbu)_n = -\frac{n_1^2 \pi^2}{d_1^2} \bbu_n \text{ for all } (n_1,n_2) \in \mathbb{N}_0^2.
    \]
\end{lemma}

\begin{proof}
    Recalling the definition of $\theta$ in \eqref{def : definition of theta}, we have
    \begin{align*}
       \theta  =  \int_{\Omega} (\bu+ 2c \overline{w})\partial_{x_1}^2\bu + \int_{\Omega} (\bu+ 2c \overline{w})\partial_{x_1}^2(\tilde{u} - \bu) + \int_{\Omega} (\tilde{u} - \overline{ u} + 2c \tilde{w} - 2c \overline{w})\partial_{x_1}^2\tilde{u},
    \end{align*}
    where $\overline{w}\in \mathcal{H}$ is an approximation of $\tilde{w}$. In particular, employing Cauchy-Schwarz inequality, we have for the second term
    \begin{align*}
        \left| \int_{\Omega} (\bu+ 2c \overline{w})\partial_{x_1}^2(\tilde{u} - \bu)\right| &\leq \|\bu+ 2c \overline{w}\|_2 \|\partial_{x_1}^2(\tilde{u} - \bu)\|_2\\
        &\leq \|\bu+ 2c \overline{w}\|_2  \|\partial_{x_1}^2 \mathbb{L}^{-1}\|_2 \|\mathbb{L}(\tilde{u} - \bu)\|_2\\
        &\leq  \|\bu+ 2c \overline{w}\|_2  \|\partial_{x_1}^2 \mathbb{L}^{-1}\|_2 r_0
    \end{align*}
    using that $\|\mathbb{L}(\tilde{u} - \bu)\|_2 = \|\tilde{u} - \bu\|_{\mathcal{H}} \leq r_0.$ Applying the Fourier transform, we have for the second norm on the right-hand side that
    \begin{align*}
        \|\partial_{x_1}^2 \mathbb{L}^{-1}\|_2 \leq \sup_{\xi \in \R^2} \frac{(2\pi\xi_1)^2}{l(\xi)} \leq \frac{1}{2-c^2}.
    \end{align*}
   Furthermore, we have for the third term of $\theta$ that
    \begin{align*}
        \left|\int_{\Omega} (\tilde{u} - \overline{ u} + 2c \tilde{w} - 2c \overline{w})\partial_{x_1}^2\tilde{u}\right| &\leq \|\tilde{u} - \overline{ u} + 2c \tilde{w} - 2c \overline{w}\|_2 \|\partial_{x_1}^2\tilde{u}\|_2\\
        &\leq (\|\tilde{u} - \overline{ u}\|_2 + 2c \|\tilde{w} -  \overline{w}\|_2) (\|\partial_{x_1}^2\bu\|_2  + \|\partial_{x_1}^2(\tilde{u}-\bu)\|_2)\\
        &\leq \left(\kappa_1 r_0 + 2c \|\tilde{w} -  \overline{w}\|_2\right) \left(\|\partial_{x_1}^2\bu\|_2  + \tfrac{r_0}{2-c^2} \right).
    \end{align*}
    It remains to estimate $ \|\tilde{w} -  \overline{w}\|_2$. By definition of $\tilde{w}$ in \eqref{def : definition of tilde v}, we have
    \begin{align*}
        \|\tilde{w}-\overline{w}\|_2 = \| -2c D\mathbb{F}(\tilde{u}) \partial_{x_1}^2 \tilde{u} - \overline{w}\|_2 \leq 2c \|D\mathbb{F}(\tilde{u})^{-1}\|_2 \|\partial_{x_1}\tilde{u} + \tfrac{1}{2c} D\mathbb{F}(\tilde{u}) \overline{w}\|_2.  
    \end{align*}
    For the second norm on the right-hand side
    \begin{align*}
        \|\partial_{x_1}\tilde{u} + \frac{1}{2c} D\mathbb{F}(\tilde{u}) \overline{w}\|_2 \leq \|\partial_{x_1}\bu + \frac{1}{2c} D\mathbb{F}(\bu) \overline{w}\|_2 + \left|\partial_{x_1}^2(\bu-\tilde{u})\right\|_2 + \frac{1}{2c}\|D\mathbb{F}(\bu) \overline{w} - D\mathbb{F}(\tilde{u}) \overline{w}\|_2.
    \end{align*}
   For the last term, we estimate
    \begin{align*}
        \|D\mathbb{F}(\bu) \overline{w} - D\mathbb{F}(\tilde{u}) \overline{w}\|_2 = \|(e^{\bu}-e^{\tilde{u}})\overline{w}\|_2 \leq \|1 - e^{\tilde{u}-\bu}\|_2 \|e^{\bu}\overline{w}\|_\infty\\
        \leq (1-e^{\kappa_2r_0}) \|e^{\bbu}\overline{W}\|_1.
    \end{align*}
    Finally, using $\bu = \gamma^\dagger(\bbu)$ and $\overline{w} = \gamma^\dagger(\overline{W})$, we apply Parseval's identity to get
    \begin{align*}
        \int_{\Omega} (\bu+ 2c \overline{w})\partial_{x_1}^2\bu = |\om| \left((\bbu+ 2c \overline{W}), \partial_{x_1}^2\bbu\right)_2,
    \end{align*}
    where $(\partial_{x_1}^2\bbu)_n = -\frac{n_1^2 \pi^2}{d_1^2} \bbu_n$ for all $(n_1,n_2) \in \mathbb{N}_0^2$. After algebraic manipulation, this concludes the proof.
\end{proof}

Using the previous result, we obtain a computer-assisted strategy for the enclosure of the value of $\theta$. Indeed, the constant $\epsilon$ depends on finite-dimensional computations, which can be handled using rigorous numerics (cf. \cite{julia_our_code}). In particular, this allows us to conclude about the stability of $\tilde{u}$. 

\begin{remark}
    We emphasize that the orbital stability established here is restricted to the subspace of perturbations that are symmetric with respect to $x_2$. While we explicitly decomposed the space $L^2(\Omega)$ into even and odd subspaces with respect to the $x_1$-direction (see \eqref{def : Hc and Hs}), our choice of ansatz in \eqref{eq:ansatz} restricts the analysis to functions that are even in $x_2$. Consequently, the analysis does not rule out the potential existence of unstable functions that are antisymmetric in $x_2$, which may still satisfy the boundary conditions \eqref{eq : neumann boundary conditions}. To prove stability against the full class of admissible perturbations satisfying \eqref{eq : neumann boundary conditions}, one can employ a domain doubling technique in the $x_2$-direction and mirror and shift the solution, as demonstrated in \cite{BredenBerg}.
\end{remark}
\section{Results}\label{sec : results}

In this section, we present computer-assisted proofs for the existence of solutions to \eqref{eq : suspension bridge equation} with boundary conditions \eqref{eq : neumann boundary conditions}. To do so, we begin with approximate solutions obtained by Newton's method and initialized with carefully constructed guesses based on approximations in the periodic setting described in \cite{Aalst2024}. To extend the initial guesses in the periodic case to the infinite setting, we project them into the kernel of matrix $\mathcal{T}$, as described in Section \ref{sec:constructionu0}. 

By combining these new approximations with the estimates established in the previous section, we obtain a validated proof of existence for solutions to the suspension bridge equation within explicitly controlled error bounds. This leads to the following theorem statement:
\begin{theorem}\label{thm:proof}
    For $c=1.2$, $d=(\frac{\pi}{0.06},\frac{\pi}{0.24})$, $N_0=(150,80)$, $N=(40,40)$, there exists a solution $\tilde{u} \in \mathcal{H} \cap C^\infty(\Omega)$ to \eqref{eq : suspension bridge equation} on $\Omega$ with boundary conditions  \eqref{eq : neumann boundary conditions} in the $x_2$-direction such that 
    \begin{align*}
        \|\tilde{u}-\bu\|_\mathcal{H} \leq 9.5 \cdot 10^{-8},
    \end{align*}
    where $\bu$ is a numerically computed approximation visualized in Figure \ref{fig:proof}. Moreover, $D\mathbb{F}(\tilde{u}) : \mathcal{H} \to L^2(\Omega)$ possesses exactly one negative eigenvalue and $\theta >0$, where $\theta$ is given in \eqref{def : definition of theta}. Finally, $\tilde{u}$ is orbitally stable.
\end{theorem}
\begin{proof}
    Recall that $\nu$ and $N^\mathrm{FFT}$ are defined in Section \ref{sec:nonlin}. We choose $\nu=(1.1,1.1)$ and $N^\mathrm{FFT}=(512,512)$. Finitely many cosine coefficients of the approximate solution $\bu$, visualized in Figure \ref{fig:proof}, can be found in the code in \cite{julia_our_code}. We report the values for the bounds in floating point numbers displaying finitely many decimals. We have
    \begin{align*}
    \mathcal{Y}_0=6.0005 \cdot 10^{-8}, \quad \mathcal{Z}_1=3.6608 \cdot 10^{-1}, \quad \mathcal{Z}_2=22.7926.
\end{align*}
Letting $\tilde{u}$ denote the zero of $\mathbb{F}$, we apply Theorem \ref{thm:existence} to obtain the error bound
\begin{align*}
    \|\tilde{u}-\bu\|_\mathcal{H} \leq 9.4657 \cdot 10^{-8}.
\end{align*}
Note that the regularity of $\tilde{u}$ (and the fact that it is infinitely differentiable) is a direct consequence of Proposition 2.5 in \cite{Cadiot2023}. 
Concerning the enclosure of the spectrum, we apply the analysis derived in Section \ref{sec : enclosure of spectrum}. In particular, we apply Lemma \ref{lem : enclosure spectrum} and prove that $D\mathbb{F}_e(\tilde{u})$ possesses exactly one negative eigenvalue and the rest of its spectrum is purely positive. Indeed, we prove that one Gershgorin interval $[\lambda_0-\epsilon_0, \lambda_0 + \epsilon_0] \subset [-0.1743, -0.1286]$ of Lemma \ref{lem : enclosure spectrum} is fully contained in the negative part of $\R$, and the rest of the intervals are disjoint from $[\lambda_0-\epsilon_0, \lambda_0 + \epsilon_0]$ and fully located in the positive part of $\R$. 

Moreover, for $D\mathbb{F}_o(\tilde{u})$, we obtain an interval $[\lambda_0-\epsilon_0, \lambda_0 + \epsilon_0] \subset [-0.0364, 0.0119]$ (using the notation of Lemma \ref{lem : enclosure spectrum}) which is disjoint from the rest of the intervals and which contains the number $0$. We also prove that the rest of the spectrum is purely positive. Consequently, there exists an eigenvalue $\lambda$, of undetermined sign, in $[\lambda_0-\epsilon_0, \lambda_0 + \epsilon_0]$. However, we prove that $\lambda = 0$ is an eigenvalue. Indeed, since $\tilde{u}$ is a zero of $\mathbb{F}$, we have
\begin{align}
   0 =  \partial_{x_1}\mathbb{F}(\tilde{u}) = D\mathbb{F}(\tilde{u}) \partial_{x_1} \tilde{u}.
\end{align}
This implies that $\partial_{x_1} \tilde{u}$ is in the kernel of $D\mathbb{F}(\tilde{u})$. Moreover, since $\tilde{u} \in \mathcal{H}$, then $\partial_{x_1} \tilde{u} \in L^2_o$, using the notation of \eqref{def : Hc and Hs}. In particular, since $\tilde{u} \neq 0$, we obtain that $0 \in Eig(D\mathbb{F}_o(\tilde{u}))$. This implies that $\lambda = 0.$ Concerning the sign of $\theta$, we use Lemma \ref{lem : enclosure of theta} and prove that $\theta >0$. Finally, we conclude the proof using Lemma \ref{lem : conditions for stability}.
\end{proof}
The code used to obtain the above result, as well as the results presented later in this section, can be found in \cite{julia_our_code}. The computations were carried out in Julia (v1.10.4) on an Apple M1 Pro CPU with $32$ GB RAM. We heavily relied on interval arithmetic using the packages IntervalArithmetic.jl (\cite{IntervalArithmetic.jl}) and RadiiPolynomial.jl (\cite{RadiiPolynomial.jl}) to ensure that the computations are rigorous.

The method developed in this paper is versatile and can also be applied to rigorously prove existence of solutions on infinite strips for other parameter values (e.g. different $c$, $d_1$, $d_2$) and with alternative spatial patterns. Two illustrative examples are presented in Figure~\ref{fig:altproofs}. In contrast to the approximation shown in Figure~\ref{fig:proof}, which contains a single prominent peak at the bottom, the approximations in Figure~\ref{fig:altproofs} show an alternative spatial pattern characterized by two dominant peaks at the bottom. In Figure~\ref{fig:c13}, the parameter $c$ is slightly increased relative to the value used in Theorem~\ref{thm:proof}, while in Figure~\ref{fig:c08}, it is decreased. The plots indicate that reducing the value of $c$ leads to an increase in the amplitude of the resulting wave patterns.

\begin{figure}[htp]
    \centering
    \subfloat[Approximate solution for $c=1.3$.]{%
        \includegraphics[width=0.5\textwidth]{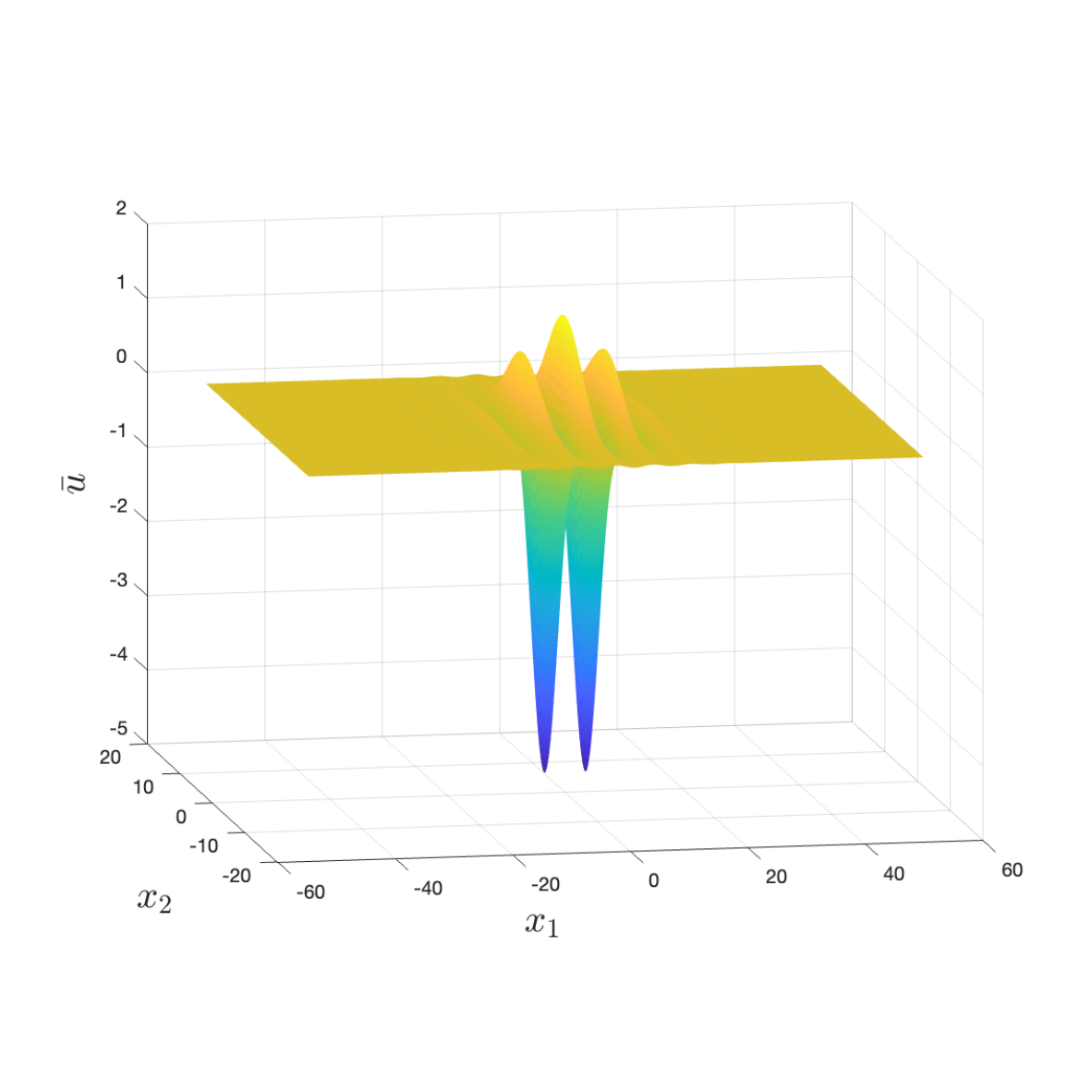}%
        \label{fig:c13}%
        }%
    \hfill%
    \subfloat[Approximate solution for $c=0.8$.]{%
        \includegraphics[width=0.5\textwidth]{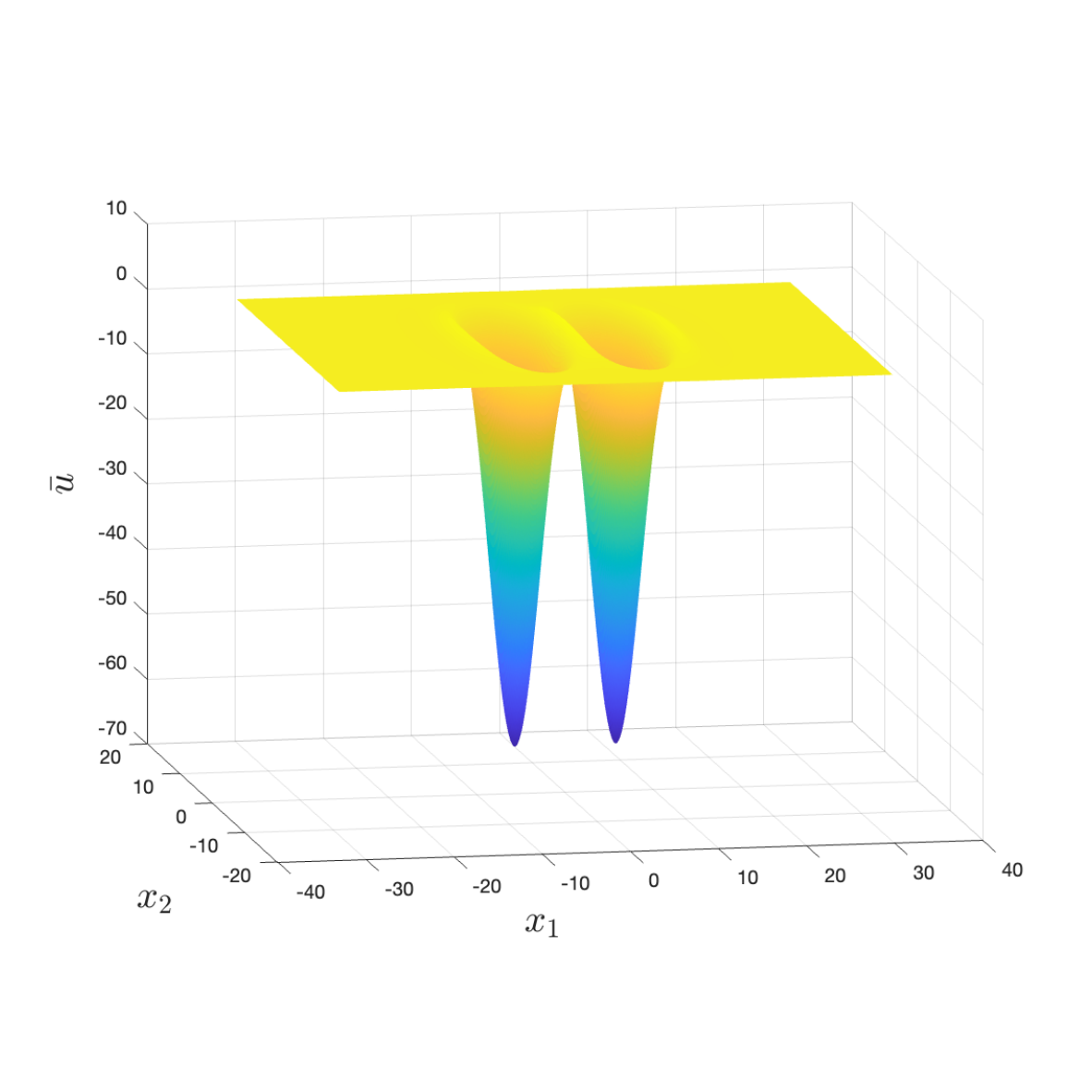}%
        \label{fig:c08}%
        }%
    \caption{Visualization of approximate two-peak solutions $\bar{u}$ to \eqref{eq : suspension bridge equation}. The approximations are truncated to a finite domain in these plots. }\label{fig:altproofs} 
\end{figure}

The existence theorems and corresponding proofs, with stability analysis, of the two alternative patterns are given below.

\begin{theorem}\label{thm:proofc13}
    For $c=1.3$, $d=(\frac{\pi}{0.06},\frac{\pi}{0.2})$, $N_0=(100,60)$, $N=(60,60)$, there exists a solution $\tilde{u} \in \mathcal{H}\cap C^\infty(\Omega)$ to \eqref{eq : suspension bridge equation} on $\Omega$ with boundary conditions \eqref{eq : neumann boundary conditions} in the $x_2$-direction such that 
    \begin{align*}
        \|\tilde{u}-\bu\|_\mathcal{H} \leq 3.1 \cdot 10^{-6},
    \end{align*}
    where $\bu$ is a numerically computed approximation visualized in Figure \ref{fig:c13}. Moreover, $D\mathbb{F}(\tilde{u}) : \mathcal{H}  \to L^2(\Omega)$ possesses exactly two negative eigenvalues and $\theta >0$, where $\theta$ is given in \eqref{def : definition of theta}. Finally, $\tilde{u}$ is orbitally unstable.
\end{theorem}
\begin{proof}
    We choose $\nu=(1.1,1.1)$ and $N^\mathrm{FFT}=(512,512)$. We have
    \begin{align*}
    \mathcal{Y}_0=2.4835 \cdot 10^{-6}, \quad \mathcal{Z}_1=1.792 \cdot 10^{-1}, \quad \mathcal{Z}_2=40.3926.
\end{align*}
Letting $\tilde{u}$ denote the zero of $\mathbb{F}$, we apply Theorem \ref{thm:existence} to obtain the error bound
\begin{align*}
    \|\tilde{u}-\bu\|_\mathcal{H} \leq 3.0256 \cdot 10^{-6}.
\end{align*}
The stability part  follows similarly as the one presented in the proof of Theorem \ref{thm:proof}. In this case, we provide the existence of exactly two negative eigenvalues by proving that two Gershgorin intervals of Lemma \ref{lem : enclosure spectrum} are in the negative part of $\R$, disjoint from one another, and disjoint from the rest of the intervals. 
\end{proof}
The next theorem and its proof are very similar to the previous ones. However, the stability analysis is more involved, so we choose to present the full proof again.
\begin{theorem}\label{thm:proofc08}
    For $c=0.8$, $d=(\frac{\pi}{0.1},\frac{\pi}{0.2})$, $N_0=(80,60)$, $N=(80,60)$, there exists a solution $\tilde{u} \in \mathcal{H}\cap C^\infty(\Omega)$ to \eqref{eq : suspension bridge equation} on $\Omega$ with boundary conditions \eqref{eq : neumann boundary conditions} in the $x_2$-direction such that 
    \begin{align*}
        \|\tilde{u}-\bu\|_\mathcal{H} \leq 4.0 \cdot 10^{-5},
    \end{align*}
    where $\bu$ is a numerically computed approximation visualized in Figure \ref{fig:c08}.   Moreover, $D\mathbb{F}(\tilde{u}) : \mathcal{H} \to L^2(\Omega)$ possesses exactly two negative eigenvalues and $\theta >0$, where $\theta$ is given in \eqref{def : definition of theta}. Finally, $\tilde{u}$ is orbitally unstable.
\end{theorem}
\begin{proof}
    We choose $\nu=(1.1,1.1)$ and $N^\mathrm{FFT}=(512,512)$. We have
    \begin{align*}
    \mathcal{Y}_0=3.3557 \cdot 10^{-5}, \quad \mathcal{Z}_1=1.4034 \cdot 10^{-1}, \quad \mathcal{Z}_2=65.7018.
\end{align*}
Letting $\tilde{u}$ denote the zero of $\mathbb{F}$, we apply Theorem \ref{thm:existence} to obtain the error bound
\begin{align*}
    \|\tilde{u}-\bu\|_\mathcal{H} \leq 3.9094 \cdot 10^{-5}.
\end{align*}
For the stability part, we obtain $\lambda_{min} \approx -15$, using the notation of  Lemma \ref{lem : lambda min}. Then, we choose $t > \lambda_{min}$ in Lemma \ref{lem : enclosure spectrum} and compute the bounds of the lemma to compute the Gershgorin intervals $\overline{B_{\epsilon_j}(\lambda_j)}$. We obtained intersecting intervals and could not conclude about the localization of negative eigenvalues. However, we obtained that the union of all the intervals was contained in $(-0.4,\infty)$. This implies that we can improve our choice for $\lambda_{min}$ and choose $\lambda_{min} = -0.45$. Choosing $t = 0.5$ in  Lemma \ref{lem : lambda min} and computing the bounds again, we get tighter radii for the Gershgorin intervals and obtain disjoint intervals on the negative part of $\R$.

The stability part then follows similarly as the one presented in the proof of Theorem \ref{thm:proof}. In this case, we provide the existence of exactly two negative eigenvalues by proving that two Gershgorin intervals of Lemma \ref{lem : enclosure spectrum} are in the negative part of $\R$, disjoint from one another, and disjoint from the rest of the intervals. 
\end{proof}
The computation time to obtain the result in Theorem \ref{thm:proof} was $892$ seconds, including the stability proof. The computation times corresponding to the results in Theorems~\ref{thm:proofc13} and~\ref{thm:proofc08} were $420$ and $586$ seconds, respectively. The increased computation time associated with our first theorem arises from selecting a larger value for $N_0$, which in turn yields a tighter error bound.

\section{Acknowledgements}
 Matthieu Cadiot was supported by the ANR project CAPPS: ANR-23-CE40-0004-01 and by the FMJH :  ANR-22-EXES-0013.
 Lindsey van der Aalst was partially supported by NWO grant 613.009.132.

\section{Appendix}
In the next lemma we provide some technical results about $l_n$, where we recall from \eqref{eq:defl} that
\begin{equation*}
l_{n_2}(\xi_1) = (2\pi)^4(  \xi_1^2 +  \tilde{n}_2^2 )^2 -  c^2 (2\pi \xi_1)^2 + 1
\end{equation*}
and that $\tilde{n}_2 = \frac{n_2}{2d_2}$.

\begin{lemma}\label{lem : results on l}
We have 
    \begin{equation}\label{eq : minimum of 1 over l}
        \min_{\xi_1 \in \R} l_{n_2}(\xi_1) = \begin{cases}
           c^2\tilde{n}_2^2+ 1-\frac{c^4}{4} &\text{ if }  0 \leq n_2 \leq \frac{d_2 c}{\pi \sqrt{2}}\\
            \tilde{n}_2^4 +1 &\text{ otherwise.}
        \end{cases}
    \end{equation}
Furthermore, given $n_2 \geq \frac{d_2c}{\pi }$, we have 
\begin{align*}
    l_{n_2}(\xi_1) \geq \frac{(2\pi)^4}{2}(\tilde{n}_2^2 + \xi_1^2)^2
\end{align*}
for all $\xi \in \R$. In particular, given $N \in \mathbb{N}$ such that $N > \frac{d_2c}{\pi }$, this implies that 
\begin{align}\label{eq:tail_lnsum}
   \sum_{n_2 = N+1}^\infty \left\|\frac{1}{l_{n_2}}\right\|_2^2 \leq \frac{5 (2d_2)^7}{48 (2\pi)^7 N^6}.
\end{align}
\end{lemma}

\begin{proof}
Denoting $x = (2\pi \xi_1)^2$ and $\alpha = (2\pi\xi_2)^2$, we have
\begin{align*}
    l(\xi_1,\xi_2) = (x+\alpha)^2 - c^2 x + 1 \bydef p_\alpha(x).
\end{align*}
Studying the variations of $p_\alpha$, we have that $p_\alpha$ has a global minimum at $x = \frac{c^2}{2}-\alpha$. In particular, this implies that
\begin{align*}
    l(\xi_1,\xi_2) \geq \begin{cases}
        (2\pi\xi_2)^2c^2 +1 - \frac{c^4}{4} &\text{ if } \frac{c^2}{2} \geq (2\pi\xi_2)^2\\
        (2\pi\xi_2)^4 +1 &\text{ otherwise}.
    \end{cases}
\end{align*}
Choosing $\xi_2 = \frac{n_2}{2d_2}$, we obtain the desired result. 
To prove the remaining statement, notice that 
\begin{align*}
    l_{n_2}(\xi_1)- \frac{1}{2}((2\pi\tilde{n}_2)^2 + (2\pi\xi_1)^2)^2 &= \frac{1}{2}((2\pi\tilde{n}_2)^2+ x)^2 - c^2x +1\\
    &= \frac{1}{2}((2\pi\tilde{n}_2)^4+ 2(2\pi\tilde{n}_2)^2x +x^2) - c^2x +1
\end{align*}
where $x = (2\pi\xi_1)^2$. Consequently, we have that $ l_{n_2}(\xi_1)- \frac{1}{2}((2\pi\tilde{n}_2)^2 + (2\pi\xi_1)^2)^2 \geq 0$ if $2\pi\tilde{n}_2 \geq c$.
Now, we have
\begin{align*}
    \left\|\frac{1}{l_{n_2}}\right\|_2^2 \leq  \frac{4}{(2\pi)^8}\int_{\R} \frac{1}{(\tilde{n}_2^2 + \xi_1^2)^4}d\xi_1 = \frac{20\pi}{16(2\pi)^8 \tilde{n}_2^7} = \frac{5}{8 (2\pi\tilde{n}_2)^7}.
\end{align*}
Moreover, we get
\begin{align*}
    \sum_{n_2 = N+1}^\infty \frac{1}{\tilde{n}_2^7} = (2d_2)^7 \sum_{n_2 = N+1}^\infty \frac{1}{n_2^7} \leq (2d_2)^7 \int_{N}^\infty \frac{1}{x^7}dx =   \frac{(2d_2)^7}{6N^6}.
\end{align*}
Noticing that $2\pi \tilde{n}_2 \geq c$ is equivalent to $n_2 \geq \frac{d_2c}{\pi}$, we have 
\begin{align*}
    \sum_{n_2 = N+1}^\infty \left\|\frac{1}{l_{n_2}}\right\|_2^2 \leq  \frac{(2d_2)^7}{6N^6}\frac{5}{8(2\pi)^7}
\end{align*}
 where $N\in \mathbb{N}$ such that $N >\frac{d_2c}{\pi}.$
\end{proof}

\bibliographystyle{abbrv}
\bibliography{bibl}
    
\end{document}